\documentclass[preprint,12pt]{elsarticle}

\usepackage{amssymb}
\usepackage{amsmath}
\usepackage{amsfonts}
\usepackage{graphicx,amsmath,amsfonts,amssymb,amscd,amsthm}
\usepackage{epsfig}
\usepackage{epsf}
\biboptions{sort&compress}
\newtheorem{theorem}{Theorem}

\newtheorem{definition}{Definition}

\newtheorem{corollary}{Corollary}

\newcommand{\x}{\mbox{\boldmath $x$}}

\newcommand{\be}{\begin{equation}}
\newcommand{\ee}{\end{equation}}
\newcommand{\ba}{\begin{align}}
\newcommand{\ea}{\end{align}}

\newcommand{\bea}{\begin{array}}
\newcommand{\eea}{\end{array}}

\def\EquationsBySection{\def\theequation
{\thesection.\arabic{equation}}%
\@addtoreset{equation}{section}}

\newcommand\old[1]{}
\newcommand{\pend}{\hfill \thicklines \framebox(6.6,6.6)[l]{}}
\renewenvironment{proof}{\noindent {\it  Proof.} \rm}{\pend}

\usepackage{indentfirst}
\EquationsBySection \makeatother
\usepackage{indentfirst}
\usepackage{floatrow}
\newfloatcommand{capbtabbox}{table}[][\FBwidth]
\journal{}
\begin{document}

\begin{frontmatter}



\title{Well-posedness and regularity of mean-field backward doubly stochastic Volterra integral equations and applications to dynamic risk measures}


\author[label1]{Bixuan Yang}
\ead{bixuanyang@126.com}
\author[label2]{Jinbiao Wu}
\ead{wujinbiao@csu.edu.cn}
\author[label2]{Tiexin Guo \corref{cor1}}
\ead{tiexinguo@csu.edu.cn}
\cortext[cor1]{Corresponding author.}

\address[label1]{\small\it School of Mathematics and Statistics, Hunan First Normal University, Changsha 410205, Hunan, China}
\address[label2]{\small\it School of Mathematics and Statistics, Central South University, Changsha 410083, Hunan, China}

%

\begin{abstract}
In this paper, the theory of mean-field backward doubly stochastic
Volterra integral equations (MF-BDSVIEs) is studied.
First, we derive the well-posedness of M-solutions to MF-BDSVIEs,
and prove the comparison theorem for such a type of equations.
Furthermore, the regularity result
of the M-solution for MF-BDSVIEs is established by virtue of
Malliavin calculus. Finally, as an application of
the comparison theorem, we obtain the properties of dynamic risk measures
governed by MF-BDSVIEs.
\end{abstract}


\end{frontmatter}

\section{Introduction}
Since Pardoux and Peng \cite{Pardoux1990} first established the
well-posedness of nonlinear backward stochastic differential equations (BSDEs) in 1990,
this type of BSDEs has frequently appeared in many
practical problems such as finance \cite{Nualart2001}, insurance \cite{Sun2018}, risk management \cite{Elliott2015},
stochastic optimal control \cite{Wu2020} and stochastic differential games \cite{Yang2019}.
Especially, in order to better simulate the uncertainty of investment risks in the financial market,
dynamic risk measures based on various kinds of BSDEs came into being to evaluate and avoid risks
(see \cite{Guo2020,Ji2019,Wang2021a}).
In 1994, backward doubly stochastic differential equations (BDSDEs) as a generalization of BSDEs
were introduced by Pardoux and Peng \cite{Pardoux1994},
including both a forward It\^{o} integral and a backward It\^{o} integral.
In addition, they also established a probabilistic
representation for the solution of quasilinear stochastic partial differential equations in \cite{Pardoux1994},
which leads to scholars' much interest in studying the regularity of BSDEs by means of Malliavin calculus,
for example, see El Karoui et al. \cite{Karoui1997}, Ma and Zhang \cite{Ma2002}, Bouchard et al. \cite{Bouchard2018},
Wen and Shi \cite{Wen2019}, Bachouch and Matoussi \cite{Bachouch2020}.

Inspired by stochastic input-output production technologies \cite{Duffie1986},
Lin \cite{Lin2002} primarily studied a class of backward stochastic Volterra integral
equations (BSVIEs) to characterize the property of state variables with memories.
Subsequently, Yong \cite{Yong2006,Yong2008} carefully considered the following general BSVIE
of the form
\begin{equation*}
\begin{aligned}
Y(t)=\zeta(t)+\int_{t}^{T} f(t, s, Y(s), Z(t, s), Z(s, t))ds
-\int_{t}^{T} Z(t, s) d\overrightarrow{W}(s), ~~t\in[0, T].
\end{aligned}
\end{equation*}
Since $\int_{t}^{T} Z(t, s) d\overrightarrow{W}(s), 0\leq t \leq T$ is not a martingale,
It\^{o}'s formula no longer applies to BSVIEs. For this reason, the definition of M-solutions
was first introduced by Yong \cite{Yong2008} to prove the well-posedness and regularity results of BSVIEs.
Moreover, their applications to dynamic risk measures and stochastic optimal control systems were also investigated.
Later, Ren \cite{Ren2010} analyzed the well-posedness and stability of the adapted
M-solution to BSVIEs with jumps in Hilbert spaces.
Wang \cite{Wang2021b} obtained the existence, uniqueness, some regularity
results and the Feynman-Kac formula for extended BSVIEs. Recently,
Shi et al. \cite{Shi2020} deeply investigated a backward doubly stochastic Volterra integral
equation (BDSVIE) involving the well-posedness result, the comparison theorem and
the stochastic maximum principle.
Miao et al. \cite{Miao2021} discussed the dynamic risk measure for
anticipated BDSVIEs based on a comparison theorem.

It is well-known that mean-field stochastic systems were initiated in \cite{Kac1956} to depict
the physical phenomena, and have been used widely in practical problems. Buckdahn et al. \cite{Buckdahn2009a,Buckdahn2009b} considered
the viscosity solution problem for a kind of mean-field BSDEs corresponding to
the nonlocal partial differential equation in detail. Since then,
mean-field models have attracted many scholars to study, for instance,
the theory of general mean-field BDSDEs with continuous coefficients \cite{Li2022} and mean-variance portfolio selection
problems \cite{Shen2013}. Furthermore,
Shi et al. \cite{Shi2013} proved the existence and uniqueness to mean-field
BSVIEs via M-solutions, and explored the maximum principle for
mean-field stochastic control systems.
Wu and Hu \cite{Wu2023} were concerned with a class of BDSVIEs with McKean-Vlasov type
defined on the product probability space, and mainly studied the duality principle
through the linear Volterra integral equations.

In this paper, we focus on a new kind of mean-field BDSVIEs (MF-BDSVIEs) as follows:
\begin{equation*}
\begin{aligned}
Y(t)&=\zeta(t)+\int_{t}^{T} f(t, s, Y(s), Z(t, s), Z(s, t), \mathbb{E}[Y(s)], \mathbb{E}[Z(t, s)], \mathbb{E}[Z(s, t)])ds\\
&~~~+\int_{t}^{T} g(t, s, Y(s), Z(t, s), Z(s, t), \mathbb{E}[Y(s)], \mathbb{E}[Z(t, s)], \mathbb{E}[Z(s, t)]) d\overleftarrow{B}(s)\\
&~~~-\int_{t}^{T} Z(t, s) d\overrightarrow{W}(s), ~~~~t\in[0, T],
\end{aligned}
\end{equation*}
where $f$ and $g$ are given maps, the $d\overrightarrow{W}$-integral means a forward It\^{o} integral, and the $d\overleftarrow{B}$-integral means a backward It\^{o} integral.
According to the physical meaning of mean-field systems, the MF-BDSVIE can be regarded as the mean-square limit of an interacting particle system ($n\rightarrow\infty$):
\begin{equation*}
\begin{aligned}
Y^{i,n}(t)&=\zeta(t)+\int_{t}^{T} f(t, s, Y^{i,n}(s), Z^{i,n}(t, s), Z^{i,n}(s, t), \frac{1}{n}\sum_{i=1}^{n} Y^{i,n}(s), \frac{1}{n}\sum_{i=1}^{n} Z^{i,n} (t, s),\\
&~~~ \frac{1}{n}\sum_{i=1}^{n} Z^{i,n}(s, t))ds
+\int_{t}^{T} g(t, s, Y^{i,n}(s), Z^{i,n}(t, s), Z^{i,n}(s, t), \frac{1}{n}\sum_{i=1}^{n} Y^{i,n}(s),\\
&~~~ \frac{1}{n}\sum_{i=1}^{n} Z^{i,n}(t, s), \frac{1}{n}\sum_{i=1}^{n} Z^{i,n}(s, t)) d\overleftarrow{B}^{i}(s)
-\int_{t}^{T} Z^{i,n}(t, s) d\overrightarrow{W}^{i}(s),~~~~t\in[0, T],
\end{aligned}
\end{equation*}
where $\{W^i (\cdot)\}_{i\geq 1}$ and $\{B^i (\cdot)\}_{i\geq 1}$
are the families of mutually independent Brownian motions.
Our target is to establish the well-posedness, the comparison theorem,
the regularity result for M-solutions of the above MF-BDSVIE, and
consider their application to the dynamic risk measure.
The innovation points of the present paper are summarized as follows.
First, we introduce the mean-field strategy into the BDSVIEs such that
both the BDSVIEs studied in
\cite{Shi2020} and the mean-field BSVIEs given in \cite{Shi2013}
become a special case of the MF-BDSVIEs considered in our paper.
Moreover, because the It\^{o} formula is missing, by the definition
of M-solutions and fixed point theorems we prove the existence
and uniqueness as well as the comparison theorem for the MF-BDSVIEs.
Our method of proving the well-posedness of MF-BDSVIEs is much simpler
than the four-step method in \cite{Yong2008}.
In addition, we present the differentiability of the M-solution to the MF-BDSVIEs
by virtue of Malliavin calculus, making up for the lack of research
on regularity results in \cite{Shi2020,Wu2023}.
Finally, as an application of the comparison theorem, some properties of dynamic risk measures
governed by MF-BDSVIEs are studied,
which extend the area of applications of the MF-BDSVIEs.

The remainder of the paper is organized as follows. In Section \ref{sec2}, we
introduce some necessary notations.
In Sections \ref{sec3}-\ref{sec4}, we prove the well-posedness
and the comparison theorem of MF-BDSVIEs, respectively.
In Section \ref{sec5}, the regularity result for the solution of MF-BDSVIEs is given.
Finally, an application to the dynamic risk measure driven by MF-BDSVIEs
is presented in Section \ref{sec6}.

\section{Preliminaries}\label{sec2}
In this section, let us introduce some basis notations to be used throughout this paper.

Let $(\Omega, \mathcal{F},
\mathbb{P})$ be a complete probability space,
and $T>0$ be a given final time.
Let $\{W(t)\}_{0\leq t \leq T}$ and $\{B(t)\}_{0\leq t \leq T}$
be two mutually independent standard Brownian motions defined on
$(\Omega, \mathcal{F}, \mathbb{P})$ with values in
$\mathbb{R}^{m}$ and in $\mathbb{R}^{k}$, respectively.
Denote the class of $\mathbb{P}$-null sets of $\mathcal{F}$ by $\mathcal{N}$.
We assume that $\mathbb{F}:=\{\mathcal{F}_t ~|~ t\in [0, T] \}$ is defined by
\begin{equation*}
\begin{aligned}
\mathcal{F}_t := \mathcal{F}_t^W \vee \mathcal{F}_{t,T}^B,
\end{aligned}
\end{equation*}
where for any process $\{ \lambda (\cdot)\}$,
$\mathcal{F}_{s,t}^\lambda := \sigma\{\lambda(r)-\lambda(s); s\leq r \leq t\}\vee\mathcal{N}$ and
$\mathcal{F}_{t}^\lambda := \mathcal{F}_{0,t}^\lambda$.
It is worth noting that $\mathbb{F}:=\{\mathcal{F}_t ~|~ t\in [0, T] \}$ is neither increasing nor
decreasing, so it is not a filtration.
Suppose that
\begin{equation*}
\begin{aligned}
\Delta=\{(t,s)\in[0, T]^2~|~t\leq s\}~~and~~\Delta^c=\{(t,s)\in[0, T]^2~|~t> s\}.
\end{aligned}
\end{equation*}
We also denote by $|x|$ the Euclidean norm of a vector $x\in \mathbb{R}^{l}$,
by $\parallel A \parallel=\sqrt{Tr AA^{T}}$ ($A^{T}$ is the transpose of $A$) the norm of a $k\times r$ matrix $A$,
by $\mathbb{S}$ the Euclidean space, and by $\mathcal{B}([a, b])$ the Borel $\sigma$-field on $[a, b]$. Furthermore, we introduce the following spaces:
\begin{equation*}
\begin{aligned}
&L^2(0,T; \mathbb{S})=\left\{\xi: [0,T]\mapsto\mathbb{S}~|~\xi(\cdot)~\mbox{is}~\mathcal{B}([0, T])\mbox{-measurable},
~\int_0^T |\xi(t)|^2 dt<\infty\right\}, \\
\end{aligned}
\end{equation*}
\begin{equation*}
\begin{aligned}
&L_{\mathcal{F}_t}^{2} (\Omega; \mathbb{S})=\left\{\xi: \Omega\mapsto\mathbb{S}~|~\xi~\mbox{is} ~\mathcal{F}_t\mbox{-measurable},
~\mathbb{E}\left[|\xi|^2\right]<\infty\right\},\\
&L_{\mathcal{F}_T}^{2} (0,T; \mathbb{S})=\left\{\nu: \Omega\times[0,T]\mapsto\mathbb{S}~|~\nu(\cdot) ~\mbox{is}~\mathcal{F}_T\mbox{-measurable},
~\mathbb{E}\left[\int_{0}^{T}|\nu(t)|^2 dt\right]<\infty\right\},\\
&L_{\mathbb{F}}^{2} (0,T; \mathbb{S})=\left\{\nu: \Omega\times[0,T]\mapsto\mathbb{S}~|~\nu(t)~ \mbox{is}~\mathcal{F}_t\mbox{-measurable for all}~t\in[0, T], \right.\\
&\left.~~~~~~~~~~~~~~~~~~\mathbb{E}\left[\int_{0}^{T}|\nu(t)|^2 dt\right]<\infty\right\},\\
&L_{\mathbb{F}}^{2} (\Delta; \mathbb{S})=\left\{\lambda: \Omega\times\Delta\mapsto\mathbb{S}~|~\mbox{for any}~t\in [0, T], ~ \lambda(t,s)~\mbox{is}~\mathcal{F}_s\mbox{-measurable for all}~s\in[t, T],\right.\\
&~~~~~~~~~~~~~~~~\left.\mathbb{E}\left[\int_{0}^{T}\int_{t}^{T}|\lambda(t,s)|^2 dsdt\right]<\infty\right\},\\
&L_{\mathbb{F}}^{2} ([0,T]^2; \mathbb{S})=\left\{\lambda: \Omega\times[0,T]\times[0,T]\mapsto\mathbb{S}~|~\mbox{for any}~t\in [0, T], ~ \lambda(t,s)~\mbox{is}~\mathcal{F}_s\mbox{-measurable for all}\right.\\
&~~~~~~~~~~~~~~~~~~~~~\left.s\in[0, T],~\mathbb{E}\left[\int_{0}^{T}\int_{0}^{T}|\lambda(t,s)|^2 dsdt\right]<\infty\right\},\\
&L^{\infty} ([0,T]; L_{\mathcal{F}_T}^{2} (\Omega; \mathbb{S}))=\left\{\nu: \Omega\times[0, T]\mapsto\mathbb{S}~|~\mbox{for any}~t\in[0, T],~ \nu(t)\in L_{\mathcal{F}_T}^{2} (\Omega; \mathbb{S})~\mbox{such that}\right.\\
&\left.~~~~~~~~~~~~~~~~~~~~~~~~~~~~~~
~\sup\limits_{0\leq t \leq T} \mathbb{E}\left[|\nu(t)|^2\right] <\infty\right\}.
\end{aligned}
\end{equation*}

\section{Well-posedness of MF-BDSVIEs}\label{sec3}

In this section, we pay attention to the existence and uniqueness of the solution to a MF-BDSVIE.
More precisely, we consider the following equation:
\begin{equation}\label{3}
\begin{aligned}
Y(t)&=\zeta(t)+\int_{t}^{T} f(t, s, Y(s), Z(t, s), Z(s, t), \mathbb{E}[Y(s)], \mathbb{E}[Z(t, s)], \mathbb{E}[Z(s, t)])ds\\
&~~~+\int_{t}^{T} g(t, s, Y(s), Z(t, s), Z(s, t), \mathbb{E}[Y(s)], \mathbb{E}[Z(t, s)], \mathbb{E}[Z(s, t)]) d\overleftarrow{B}(s)\\
&~~~-\int_{t}^{T} Z(t, s) d\overrightarrow{W}(s), ~~~~t\in[0, T],
\end{aligned}
\end{equation}
where
$\zeta: \Omega\times [0, T] \mapsto \mathbb{R}^n$
is an $\mathcal{F}_T$-measurable map.
$f: \Omega\times \Delta \times (\mathbb{R}^n \times\mathbb{R}^{n\times m}\times\mathbb{R}^{n\times m})^2
\mapsto\mathbb{R}^n$ and
$g: \Omega\times \Delta \times (\mathbb{R}^n \times\mathbb{R}^{n\times m}\times\mathbb{R}^{n\times m})^2
\mapsto \mathbb{R}^{n\times k}$ are given jointly measurable maps, and have
\begin{itemize}
\item[(A1)] For any $(t,y,z,\varsigma,\overline{y},\overline{z},\overline{\varsigma})\in [0,T]\times(\mathbb{R}^n \times\mathbb{R}^{n\times
m}\times\mathbb{R}^{n\times m})^2$, the maps $s\mapsto f(t,s,y,z,\varsigma,\overline{y},\overline{z},\overline{\varsigma})$
and $s\mapsto g(t,s,y,z,\varsigma,\overline{y},\overline{z},\overline{\varsigma})$ are $\mathcal{F}_s$-measurable.
\item[(A2)]
For any $(y,z,\varsigma,\overline{y},\overline{z},\overline{\varsigma})\in (\mathbb{R}^n \times\mathbb{R}^{n\times
m}\times\mathbb{R}^{n\times m})^2$, $f(\cdot,\cdot,y,z,\varsigma,\overline{y},\overline{z},\overline{\varsigma})\in L_{\mathbb{F}}^{2} (\Delta; \mathbb{R}^n)$
and $g(\cdot,\cdot,y,z,\varsigma,\overline{y},\overline{z},\overline{\varsigma})\in L_{\mathbb{F}}^{2} (\Delta; \mathbb{R}^{n\times k})$.
\item[(A3)]
For any $(y_1,z_1,\varsigma_1,\overline{y}_1,\overline{z}_1,\overline{\varsigma}_1),
(y_2,z_2,\varsigma_2,\overline{y}_2,\overline{z}_2,\overline{\varsigma}_2)\in(\mathbb{R}^n \times\mathbb{R}^{n\times m}\times\mathbb{R}^{n\times m})^2$ and $(t,s)\in\Delta$, there exist some constants $c>0$ and $0<\alpha<\frac{1}{2(T+2)}$ such that
\begin{align*}
&~~~|f(t,s,y_1,z_1,\varsigma_1,\overline{y}_1,\overline{z}_1,\overline{\varsigma}_1)
-f(t,s,y_2,z_2,\varsigma_2,\overline{y}_2,\overline{z}_2,\overline{\varsigma}_2)|^2\\
&\leq c(|y_1-y_2|^2+\|z_1-z_2\|^2+\|\varsigma_1-\varsigma_2\|^2+|\overline{y}_1-\overline{y}_2|^2
+\|\overline{z}_1-\overline{z}_2\|^2+\|\overline{\varsigma}_1-\overline{\varsigma}_2\|^2);\\
&~~~\|g(t,s,y_1,z_1,\varsigma_1,\overline{y}_1,\overline{z}_1,\overline{\varsigma}_1)
-g(t,s,y_2,z_2,\varsigma_2,\overline{y}_2,\overline{z}_2,\overline{\varsigma}_2)\|^2\\
&\leq \alpha(|y_1-y_2|^2+\|z_1-z_2\|^2+\|\varsigma_1-\varsigma_2\|^2+|\overline{y}_1-\overline{y}_2|^2
+\|\overline{z}_1-\overline{z}_2\|^2+\|\overline{\varsigma}_1-\overline{\varsigma}_2\|^2).
\end{align*}
\end{itemize}

\begin{definition}\label{def1}
For all $S\in [0,T]$, a pair of processes $(Y(\cdot),Z(\cdot,\cdot))\in L_{\mathbb{F}}^{2} (0,T; \mathbb{R}^n)\times L_{\mathbb{F}}^{2} ([0,T]^2;$ $\mathbb{R}^{n\times m})$ is called an M-solution of MF-BDSVIE (\ref{3}) if it satisfies
Eq. (\ref{3}) in the usual It\^{o}'s sense and the following equation holds:
\begin{equation}\label{4}
\begin{aligned}
Y(t)=\mathbb{E}[Y(s)|\mathcal{F}_S]+\int_S^t Z(t,s) d\overrightarrow{W}(s),~~~a.e.~t\in [S,T].
\end{aligned}
\end{equation}
\end{definition}

Subsequently, for each $\beta>0$, we define $\mathcal{L}_\beta^2 [0,T]$ to be the Banach space
\begin{equation*}
\begin{aligned}
\mathcal{L}_\beta^2 [0,T]=L_{\mathbb{F}}^{2} (0,T; \mathbb{R}^n)\times L_{\mathbb{F}}^{2} ([0,T]^2; \mathbb{R}^{n\times m}),
\end{aligned}
\end{equation*}
equipped with the norm
\begin{equation*}
\begin{aligned}
\|(Y(\cdot),Z(\cdot,\cdot))\|_{\mathcal{L}_\beta^2 [0,T]}:=
\left\{\mathbb{E}\left[\int_0^T e^{\beta t} |Y(t)|^2 dt+\int_0^T \int_0^T  e^{\beta s}\|Z(t,s)\|^2 dsdt\right]\right\}^{1/2}.
\end{aligned}
\end{equation*}

Moreover, for any $\beta>0$, we define $\mathcal{M}_\beta^2 [0,T]$ to be the set of all pairs $(Y(\cdot),Z(\cdot,\cdot))\in
\mathcal{L}_\beta^2 [0,T]$ such that (\ref{4}) is satisfied. Under (\ref{4}), for any $(Y(\cdot),Z(\cdot,\cdot))\in
\mathcal{M}_\beta^2 [0,T]$, we obtain
\begin{equation*}
\begin{aligned}
&~~~\mathbb{E}\left[\int_0^T e^{\beta t} |Y(t)|^2 dt+\int_0^T \int_0^T  e^{\beta s}\|Z(t,s)\|^2 dsdt\right]\\
&\leq\mathbb{E}\left[\int_0^T e^{\beta t} |Y(t)|^2 dt+\int_0^T \int_t^T  e^{\beta s}\|Z(t,s)\|^2 dsdt
+\int_0^T e^{\beta t}\int_0^t  \|Z(t,s)\|^2 dsdt\right]\\
&=\mathbb{E}\left[\int_0^T e^{\beta t} |Y(t)|^2 dt+\int_0^T \int_t^T  e^{\beta s}\|Z(t,s)\|^2 dsdt\right]
+\mathbb{E}\left[\int_0^T e^{\beta t}\left(\int_0^t Z(t,s) d\overrightarrow{W}(s)\right)^2 dt\right] \\
&=\mathbb{E}\left[\int_0^T e^{\beta t} |Y(t)|^2 dt+\int_0^T \int_t^T  e^{\beta s}\|Z(t,s)\|^2 dsdt\right]
+\mathbb{E}\left[\int_0^T e^{\beta t}\left(Y(t)-\mathbb{E}[Y(s)]\right)^2 dt\right]\\
&\leq2\mathbb{E}\left[\int_0^T e^{\beta t} |Y(t)|^2 dt+\int_0^T \int_t^T  e^{\beta s}\|Z(t,s)\|^2 dsdt\right],
\end{aligned}
\end{equation*}
which means that we can adopt an equivalent norm in $\mathcal{M}_\beta^2 [0,T]$:
\begin{align*}
\|(Y(\cdot),Z(\cdot,\cdot))\|_{\mathcal{M}_\beta^2 [0,T]}:=
\left\{\mathbb{E}\left[\int_0^T e^{\beta t} |Y(t)|^2 dt+\int_0^T \int_t^T  e^{\beta s}\|Z(t,s)\|^2 dsdt\right]\right\}^{1/2}.
\end{align*}

Now we prove the existence and uniqueness of the M-solution to MF-BDSVIE (\ref{3}).

\begin{theorem}\label{th1}
If $(A1)$-$(A3)$ hold and $\zeta(\cdot)\in L_{\mathcal{F}_T}^{2} (0,T; \mathbb{R}^n)$, then there exists a unique M-solution $(Y(\cdot), Z(\cdot,\cdot))\in \mathcal{M}_\beta^2 [0,T]~(\beta>0)$
of MF-BDSVIE (\ref{3}). Moreover, there exists a constant $L>0$ such that
\begin{equation}\label{30}
\begin{aligned}
&~~~\|(Y(\cdot),Z(\cdot,\cdot))\|_{\mathcal{M}_\beta^2 [0,T]}^2 \\
&\leq L \mathbb{E}\left[\int_0^T e^{\beta t}|\zeta(t)|^2 dt+\int_0^T \int_t^T  e^{\beta s} |f_0 (t,s)|^2 dsdt +\int_0^T \int_t^T  e^{\beta s}\|g_0 (t,s)\|^2 dsdt \right],
\end{aligned}
\end{equation}
where $f_0 (t,s)=f (t,s,0,0,0,0,0,0)$ and $g_0 (t,s)=g (t,s,0,0,0,0,0,0)$.\\
For $i=1,2$, if $f_i$ and $g_i$ satisfy $(A1)$-$(A3)$, $\zeta_i (\cdot)\in L_{\mathcal{F}_T}^{2} (0,T; \mathbb{R}^n)$,
and $(Y_i (\cdot), Z_i (\cdot,\cdot))\in \mathcal{M}_\beta^2 [0,T]$ is the unique M-solution of MF-BDSVIE (\ref{3})
corresponding to $f_i$, $g_i$ and $\zeta_i$, then
\begin{equation}\label{32}
\begin{aligned}
&~~~\mathbb{E}\left[\int_0^T e^{\beta t} |Y_1 (t)-Y_2 (t)|^2 dt+\int_0^T \int_t^T  e^{\beta s}\|Z_1 (t,s)-Z_2 (t,s)\|^2 dsdt\right] \\
&\leq L \mathbb{E}\left[\int_0^T e^{\beta t}|\zeta_1 (t)-\zeta_2 (t)|^2 dt
+\int_0^T \int_t^T  e^{\beta s} |\delta f (t,s)|^2 dsdt
+\int_0^T \int_t^T  e^{\beta s}\|\delta g (t,s)\|^2 dsdt \right],
\end{aligned}
\end{equation}
where
\begin{equation*}
\begin{aligned}
\delta f(t,s)&=f_1 (t,s,Y_2(s), Z_2(t, s), Z_2(s, t), \mathbb{E}[Y_2(s)], \mathbb{E}[Z_2(t, s)], \mathbb{E}[Z_2(s, t)])\\
&~~~-f_2 (t,s,Y_2(s), Z_2(t, s), Z_2(s, t), \mathbb{E}[Y_2(s)], \mathbb{E}[Z_2(t, s)], \mathbb{E}[Z_2(s, t)]),\\
\delta g(t,s)&=g_1 (t,s,Y_2(s), Z_2(t, s), Z_2(s, t), \mathbb{E}[Y_2(s)], \mathbb{E}[Z_2(t, s)], \mathbb{E}[Z_2(s, t)])\\
&~~~-g_2 (t,s,Y_2(s), Z_2(t, s), Z_2(s, t), \mathbb{E}[Y_2(s)], \mathbb{E}[Z_2(t, s)], \mathbb{E}[Z_2(s, t)]).
\end{aligned}
\end{equation*}
\end{theorem}
\begin{proof}
For any $(y(\cdot), z(\cdot,\cdot))\in \mathcal{M}_\beta^2 [0,T]$, we investigate
the equation of the following form:
\begin{equation}\label{5}
\begin{aligned}
Y(t)&=\zeta(t)+\int_{t}^{T} f(t, s, y(s), z(t, s), z(s, t), \mathbb{E}[y(s)], \mathbb{E}[z(t, s)], \mathbb{E}[z(s, t)])ds\\ &~~~+\int_{t}^{T} g(t, s, y(s), z(t, s), z(s, t), \mathbb{E}[y(s)], \mathbb{E}[z(t, s)], \mathbb{E}[z(s, t)]) d\overleftarrow{B}(s)       \\
&~~~-\int_{t}^{T} Z(t, s) d\overrightarrow{W}(s),~t\in[0, T].
\end{aligned}
\end{equation}
By Lemma 2.5 in \cite{Shi2020} we can deduce that BDSVIE (\ref{5}) admits a unique solution $(Y(\cdot), Z(\cdot, \cdot))$ $\in L_{\mathbb{F}}^{2} (0,T; \mathbb{R}^n)\times L_{\mathbb{F}}^{2} (\Delta; \mathbb{R}^{n\times m})$.
Suppose that $Z(\cdot,\cdot)$ satisfies the relation (\ref{4}) on $\Delta^c$. Then it is obvious that $(Y(\cdot),Z(\cdot,\cdot))\in \mathcal{M}_\beta^2 [0,T]$ is an M-solution of Eq. (\ref{5}).
Hence, for each $(y(\cdot), z(\cdot,\cdot))\in \mathcal{M}_\beta^2 [0,T]$, we construct a map
\begin{equation}\label{31}
\begin{aligned}
\Gamma(y(\cdot), z(\cdot,\cdot)):=(Y(\cdot), Z(\cdot,\cdot))\in\mathcal{M}_\beta^2 [0,T].
\end{aligned}
\end{equation}
Hereafter, we only need to prove $\Gamma$ is a contractive map.

For $i=1,2$, take any $(y_i(\cdot), z_i(\cdot,\cdot))\in \mathcal{M}_\beta^2 [0,T]$,
we have $(Y_i(\cdot), Z_i(\cdot,\cdot))=\Gamma(y_i(\cdot), z_i(\cdot,\cdot))$. Denote
\begin{align*}
(\widehat{y}(\cdot), \widehat{z}(\cdot,\cdot))=(y_1(\cdot)-y_2(\cdot), z_1(\cdot,\cdot)-z_2(\cdot,\cdot)),~
(\widehat{Y}(\cdot), \widehat{Z}(\cdot,\cdot))=(Y_1(\cdot)-Y_2(\cdot), Z_1(\cdot,\cdot)-Z_2(\cdot,\cdot)).
\end{align*}
Making use of the estimate (2.2) in \cite{Shi2020} yields that
\begin{equation}\label{6}
\begin{aligned}
&~~~~\mathbb{E}\left[\int_0^T e^{\beta t} |\widehat{Y}(t)|^2 dt+\int_0^T \int_t^T  e^{\beta s}\|\widehat{Z}(t,s)\|^2 dsdt\right]\\
&\leq\frac{10}{\beta} \mathbb{E}\left[\int_0^T \int_t^T e^{\beta s}\left|f(t,s,y_1(s),z_1(t,s),z_1(s,t),\mathbb{E}[y_1(s)],\mathbb{E}[z_1(t,s)],\mathbb{E}[z_1(s,t)])\right.\right.\\
&~~~\left.\left.-f(t,s,y_2(s),z_2(t,s),z_2(s,t),\mathbb{E}[y_2(s)],\mathbb{E}[z_2(t,s)],\mathbb{E}[z_2(s,t)])\right|^2 dsdt\right]\\
&~~~+\mathbb{E}\left[\int_0^T  \int_t^T e^{\beta s} \left\|g(t,s,y_1(s),z_1(t,s),z_1(s,t),\mathbb{E}[y_1(s)],\mathbb{E}[z_1(t,s)],\mathbb{E}[z_1(s,t)])\right.\right.\\
&~~~\left.\left.-g(t,s,y_2(s),z_2(t,s),z_2(s,t),\mathbb{E}[y_2(s)],\mathbb{E}[z_2(t,s)],\mathbb{E}[z_2(s,t)])\right\|^{2} dsdt\right]\\
&~~~+\mathbb{E}\left[\int_0^T e^{\beta t} \int_t^T  \left\|g(t,s,y_1(s),z_1(t,s),z_1(s,t),\mathbb{E}[y_1(s)],\mathbb{E}[z_1(t,s)],\mathbb{E}[z_1(s,t)])\right.\right.\\
&~~~\left.\left.-g(t,s,y_2(s),z_2(t,s),z_2(s,t),\mathbb{E}[y_2(s)],\mathbb{E}[z_2(t,s)],\mathbb{E}[z_2(s,t)])\right\|^{2} dsdt\right].
\end{aligned}
\end{equation}
With the help of (\ref{4}), it follows that
\begin{equation}\label{7}
\begin{aligned}
&\mathbb{E}\left[\int_0^T \int_t^T e^{\beta s} |\widehat{y}(s)|^2 dsdt\right]
\leq T \mathbb{E}\left[\int_0^T  e^{\beta t} |\widehat{y}(t)|^2 dt\right],\\
&\mathbb{E}\left[\int_0^T \int_t^T e^{\beta s} \|\widehat{z}(s,t)\|^2 dsdt\right]
=\mathbb{E}\left[\int_0^T \int_0^s e^{\beta s} \|\widehat{z}(s,t)\|^2 dtds\right]
\\
&=\mathbb{E}\left[\int_0^T \int_0^t e^{\beta t} \|\widehat{z}(t,s)\|^2 dsdt\right]
=\mathbb{E}\left[\int_0^T e^{\beta t} \left(\int_0^t \|\widehat{z}(t,s)\|^2 d\overrightarrow{W}(s)\right)^2 dt\right] \\
&=\mathbb{E}\left[\int_0^T e^{\beta t} \left|\widehat{y}(t)-\mathbb{E}[\widehat{y}(t)]\right|^2 dt\right]
\leq\mathbb{E}\left[\int_0^T e^{\beta t} |\widehat{y}(t)|^2 dt\right],
\end{aligned}
\end{equation}
\begin{equation}\label{53}
\begin{aligned} &\mathbb{E}\left[\int_0^T \int_t^T e^{\beta s} |\mathbb{E}[\widehat{y}(s)]|^2 dsdt\right]
\leq T \mathbb{E}\left[\int_0^T  e^{\beta t} |\mathbb{E}[\widehat{y}(t)]|^2 dt\right],\\
&\mathbb{E}\left[\int_0^T \int_t^T e^{\beta s} \|\mathbb{E}[\widehat{z}(s,t)]\|^2 dsdt\right]
\leq\mathbb{E}\left[\int_0^T e^{\beta t} |\mathbb{E}[\widehat{y}(t)]|^2 dt\right].
\end{aligned}
\end{equation}
It then follows from $(A3)$, (\ref{7})-(\ref{53}) and Jensen's inequality that
\begin{equation}\label{8}
\begin{aligned}
&~~~\mathbb{E}\left[\int_0^T \int_t^T e^{\beta s}\left|f(t,s,y_1(s),z_1(t,s),z_1(s,t),\mathbb{E}[y_1(s)],\mathbb{E}[z_1(t,s)],\mathbb{E}[z_1(s,t)])\right.\right.\\
&~~~\left.\left.-f(t,s,y_2(s),z_2(t,s),z_2(s,t),\mathbb{E}[y_2(s)],\mathbb{E}[z_2(t,s)],\mathbb{E}[z_2(s,t)])\right|^2 dsdt\right]\\
&\leq c \mathbb{E}\left[\int_0^T \int_t^T e^{\beta s} \left( |\widehat{y}(s)|^2+ \|\widehat{z}(t,s)\|^2+\|\widehat{z}(s,t)\|^2
+\|\mathbb{E}[\widehat{y}(s)]\|^2+ \|\mathbb{E}[\widehat{z}(t,s)]\|^2\right.\right.\\
&~~~\left.\left.+\|\mathbb{E}[\widehat{z}(s,t)]\|^2 \right)    dsdt\right]\\
&\leq 2c(T+1)\mathbb{E}\left[\int_0^T e^{\beta t} |\widehat{y}(t)|^2 dt\right]+2c\mathbb{E}\left[\int_0^T \int_t^T e^{\beta s} \|\widehat{z}(t,s)\|^2 dsdt\right].
\end{aligned}
\end{equation}
Similar to (\ref{8}), we obtain
\begin{equation}\label{9}
\begin{aligned}
&~~~\mathbb{E}\left[\int_0^T  \int_t^T e^{\beta s} \left\|g(t,s,y_1(s),z_1(t,s),z_1(s,t),\mathbb{E}[y_1(s)],\mathbb{E}[z_1(t,s)],\mathbb{E}[z_1(s,t)])\right.\right.\\
&~~~\left.\left.-g(t,s,y_2(s),z_2(t,s),z_2(s,t),\mathbb{E}[y_2(s)],\mathbb{E}[z_2(t,s)],\mathbb{E}[z_2(s,t)])\right\|^{2} dsdt\right]\\
&\leq \alpha \mathbb{E}\left[\int_0^T \int_t^T e^{\beta s} \left( |\widehat{y}(s)|^2+ \|\widehat{z}(t,s)\|^2+\|\widehat{z}(s,t)\|^2
+\|\mathbb{E}[\widehat{y}(s)]\|^2+ \|\mathbb{E}[\widehat{z}(t,s)]\|^2\right.\right.\\
&~~~\left.\left.+\|\mathbb{E}[\widehat{z}(s,t)]\|^2 \right)    dsdt\right]\\
&\leq 2\alpha(T+1)\mathbb{E}\left[\int_0^T e^{\beta t} |\widehat{y}(t)|^2 dt\right]+2\alpha\mathbb{E}\left[\int_0^T \int_t^T e^{\beta s} \|\widehat{z}(t,s)\|^2 dsdt\right].
\end{aligned}
\end{equation}
Noting the following inequalities:
\begin{equation*}
\begin{aligned}
&\mathbb{E}\left[\int_0^T e^{\beta t} \int_t^T |\widehat{y}(s)|^2 dsdt\right]
=\frac{1}{\beta}\mathbb{E}\left[\int_0^T  \left(\int_t^T |\widehat{y}(s)|^2 ds\right)de^{\beta t}\right]\\
&=\frac{1}{\beta}\mathbb{E}\left[\int_0^T e^{\beta t} |\widehat{y}(t)|^2 dt\right]-\frac{1}{\beta}\mathbb{E}\left[\int_0^T |\widehat{y}(t)|^2 dt\right]
\leq\frac{1}{\beta}\mathbb{E}\left[\int_0^T e^{\beta t} |\widehat{y}(t)|^2 dt\right],\\
&\mathbb{E}\left[\int_0^T e^{\beta t}\int_t^T  \|\widehat{z}(s,t)\|^2 dsdt\right]
\leq\mathbb{E}\left[\int_0^T \int_t^T e^{\beta s} \|\widehat{z}(s,t)\|^2 dsdt\right]
\leq\mathbb{E}\left[\int_0^T e^{\beta t} |\widehat{y}(t)|^2 dt\right],\\
&\mathbb{E}\left[\int_0^T e^{\beta t} \int_t^T |\mathbb{E}[\widehat{y}(s)]|^2 dsdt\right]
\leq\frac{1}{\beta}\mathbb{E}\left[\int_0^T e^{\beta t} |\mathbb{E}[\widehat{y}(t)]|^2 dt\right],\\
&\mathbb{E}\left[\int_0^T e^{\beta t}\int_t^T  \|\mathbb{E}[\widehat{z}(s,t)]\|^2 dsdt\right]
\leq\mathbb{E}\left[\int_0^T e^{\beta t} |\mathbb{E}[\widehat{y}(t)]|^2 dt\right],
\end{aligned}
\end{equation*}
we deduce that
\begin{equation}\label{10}
\begin{aligned}
&~~~\mathbb{E}\left[\int_0^T e^{\beta t} \int_t^T  \left\|g(t,s,y_1(s),z_1(t,s),z_1(s,t),\mathbb{E}[y_1(s)],\mathbb{E}[z_1(t,s)],\mathbb{E}[z_1(s,t)])\right.\right.\\
&~~~\left.\left.-g(t,s,y_2(s),z_2(t,s),z_2(s,t),\mathbb{E}[y_2(s)],\mathbb{E}[z_2(t,s)],\mathbb{E}[z_2(s,t)])\right\|^{2} dsdt\right]\\
&\leq\alpha\mathbb{E}\left[\int_0^T e^{\beta t} \int_t^T \left(|\widehat{y}(s)|^2+ \|\widehat{z}(t,s)\|^2+\|\widehat{z}(s,t)\|^2
+\|\mathbb{E}[\widehat{y}(s)]\|^2+ \|\mathbb{E}[\widehat{z}(t,s)]\|^2\right.\right.\\
&~~~\left.\left.+\|\mathbb{E}[\widehat{z}(s,t)]\|^2 \right)    dsdt\right]\\
&\leq 2\left(\frac{\alpha}{\beta}+\alpha\right)\mathbb{E}\left[\int_0^T e^{\beta t} |\widehat{y}(t)|^2 dt\right]
+2\alpha\mathbb{E}\left[\int_0^T \int_t^T e^{\beta s}\|\widehat{z}(t,s)\|^2 dsdt\right].
\end{aligned}
\end{equation}
Substituting (\ref{8})-(\ref{10}) into (\ref{6}) implies that
\begin{equation}\label{16}
\begin{aligned}
&~~~~\mathbb{E}\left[\int_0^T e^{\beta t} |\widehat{Y}(t)|^2 dt+\int_0^T \int_t^T  e^{\beta s}\|\widehat{Z}(t,s)\|^2 dsdt\right]\\
&\leq\left[\frac{20c(T+1)+2\alpha}{\beta}+2\alpha(T+2) \right]\mathbb{E}\left[\int_0^T e^{\beta t} |\widehat{y}(t)|^2 dt\right]\\
&~~~+\left(\frac{20c}{\beta}+4\alpha\right)\mathbb{E}\left[\int_0^T \int_t^T e^{\beta s}\|\widehat{z}(t,s)\|^2 dsdt\right]\\
&\leq\gamma\mathbb{E}\left[\int_0^T e^{\beta t} |\widehat{y}(t)|^2 dt+\int_0^T \int_t^T e^{\beta s}\|\widehat{z}(t,s)\|^2 dsdt\right],
\end{aligned}
\end{equation}
where $\gamma=\frac{20c(T+1)+2\alpha}{\beta}+2\alpha(T+2)$. By $0<\alpha<\frac{1}{2(T+2)}$,
if $\beta>\frac{20c(T+1)+2\alpha}{1-2\alpha(T+2)}$, then it is clear that $\Gamma$ is a contractive map on $\mathcal{M}_\beta^2 [0,T]$.
From the fixed point theorem, one can see that there exists a unique fixed point $(Y(\cdot), Z(\cdot,\cdot))$
in $\Gamma$ such that
$\Gamma(Y(\cdot), Z(\cdot,\cdot))=(Y(\cdot), Z(\cdot,\cdot))$.
Thus MF-BDSVIE (\ref{3}) admits a unique M-solution $(Y(\cdot), Z(\cdot,\cdot))\in \mathcal{M}_\beta^2 [0,T]$.
Similarly, the estimates (\ref{30})-(\ref{32}) can also be inferred from (2.2) in \cite{Shi2020}.
\end{proof}

Next, let us research the special case of Theorem \ref{th1}.
\begin{corollary}\label{cor1}
Let $f=f(t,s,y,z,\overline{y},\overline{z})$ and $g=g(t,s,y,z,\overline{y},\overline{z})$. Under $(A1)$-$(A3)$,
for each $\zeta(\cdot)\in L_{\mathcal{F}_T}^{2} (0,T; \mathbb{R}^n)$, the following MF-BDSVIE:
\begin{equation*}
\begin{aligned}
Y(t)&=\zeta(t)+\int_{t}^{T} f(t, s, Y(s), Z(t, s), \mathbb{E}[Y(s)], \mathbb{E}[Z(t, s)])ds\\
&~~~+\int_{t}^{T} g(t, s, Y(s), Z(t, s), \mathbb{E}[Y(s)], \mathbb{E}[Z(t, s)]) d\overleftarrow{B}(s)
-\int_{t}^{T} Z(t, s) d\overrightarrow{W}(s), ~~t\in[0, T],
\end{aligned}
\end{equation*}
has a unique solution $(Y(\cdot), Z(\cdot,\cdot))\in L_{\mathbb{F}}^{2} (0,T; \mathbb{R}^n)\times L_{\mathbb{F}}^{2} (\Delta; \mathbb{R}^{n\times m})$.
\end{corollary}

\section{Comparison theorem of MF-BDSVIEs}\label{sec4}

In this section, we are going to present a comparison theorem for MF-BDSVIEs allowing the
one-dimensional case.

For $i=1,2$, we consider the following MF-BDSVIE:
\begin{equation}\label{13}
\begin{aligned}
Y_i (t)&=\zeta_i (t)+\int_{t}^{T} f_i (t, s, Y_i (s), Z_i (t, s), \mathbb{E}[Y_i (s)])ds\\
&~~~+\int_{t}^{T} g(t, s, Y_i (s), Z_i (t, s), \mathbb{E}[Y_i (s)]) d\overleftarrow{B}(s)
-\int_{t}^{T} Z_i (t, s) d\overrightarrow{W}(s), ~~t\in[0, T].
\end{aligned}
\end{equation}
Since the generators $f_i$ and $g$ are independent of $Z_i (s,t)$ and $\mathbb{E}[Z_i (s,t)]$,
the definition of M-solutions is no more needed.
It is well-known that suppose $(A1)$-$(A3)$ hold, for any $\zeta_i (\cdot)$ $\in L_{\mathcal{F}_T}^{2} (0,T; \mathbb{R})$,
there is a unique solution $(Y_i (\cdot), Z_i (\cdot,\cdot))\in L_{\mathbb{F}}^{2} (0,T; \mathbb{R})\times L_{\mathbb{F}}^{2} (\Delta; \mathbb{R})$ of MF-BDSVIE (\ref{13}) according to Corollary \ref{cor1}.

\begin{theorem}\label{th3}
For $i=1,2$, let $(A1)$-$(A3)$ hold for the maps $f_i$ and $g$. Suppose that $\overline{f}=\overline{f}(t,s,y,z,\overline{y})$ satisfies $(A1)$-$(A3)$ such that $y\mapsto \overline{f}(t,s,y,z,\overline{y})$ and $\overline{y}\mapsto \overline{f}(t,s,y,z,\overline{y})$ are nondecreasing, respectively, and
$$f_1 (t,s,y,z,\overline{y})\leq \overline{f}(t,s,y,z,\overline{y})\leq f_2 (t,s,y,z,\overline{y}),$$
$$\forall (t,y,z,\overline{y})\in [0,s]\times\mathbb{R}\times\mathbb{R}\times\mathbb{R},~a.s., a.e.~s\in[0, T].$$
For each $\zeta_i (\cdot)\in L_{\mathcal{F}_T}^{2} (0,T; \mathbb{R})$, if $\zeta_1 (t)\leq\zeta_2 (t), a.s.~t\in[0, T]$, then
$$Y_1 (t)\leq Y_2 (t),~a.s., a.e.~t\in[0, T].$$
\end{theorem}
\begin{proof}
Take $\overline{\zeta} (\cdot)\in L_{\mathcal{F}_T}^{2} (0,T; \mathbb{R})$ such that
$$\zeta_1 (t) \leq \overline{\zeta} (t) \leq \zeta_2 (t),~a.s.~t\in [0,T].$$
Suppose that $(\overline{Y}(\cdot), \overline{Z}(\cdot, \cdot))\in L_{\mathbb{F}}^{2} (0,T; \mathbb{R})\times L_{\mathbb{F}}^{2} (\Delta; \mathbb{R})$ is the unique solution of the following equation:
\begin{align*}
\overline{Y} (t)&=\overline{\zeta} (t)+\int_{t}^{T} \overline{f} (t, s, \overline{Y} (s), \overline{Z} (t,s), \mathbb{E}[\overline{Y} (s)])ds+\int_{t}^{T} g(t, s, \overline{Y} (s), \overline{Z} (t,s), \mathbb{E}[\overline{Y} (s)]) d\overleftarrow{B}(s)\\
&~~~-\int_{t}^{T} \overline{Z}(t, s) d\overrightarrow{W}(s), ~~~~t\in[0, T].
\end{align*}
Let $\widetilde{Y}_0 (\cdot)=Y_2 (\cdot)$. From Theorem 3.3 in \cite{Shi2020}, we denote
$(\widetilde{Y}_1 (\cdot), \widetilde{Z}_1 (\cdot,\cdot))\in L_{\mathbb{F}}^{2} (0,T; \mathbb{R})\times L_{\mathbb{F}}^{2} (\Delta; \mathbb{R})$ by the unique solution of the following equation:
\begin{align*}
\widetilde{Y}_1 (t)&=\overline{\zeta} (t)+\int_{t}^{T} \overline{f} (t, s, \widetilde{Y}_1 (s), \widetilde{Z}_1 (t,s), \mathbb{E}[\widetilde{Y}_0 (s)])ds\\
&~~~+\int_{t}^{T} g(t, s, \widetilde{Y}_1 (s), \widetilde{Z}_1 (t,s), \mathbb{E}[\widetilde{Y}_0 (s)]) d\overleftarrow{B}(s)-\int_{t}^{T} \widetilde{Z}_1 (t, s) d\overrightarrow{W}(s), ~~~~t\in[0, T].
\end{align*}
Notice that
\begin{align*}
\left\{\begin{array}{lcl} \overline{f}(t,s,y,z,\mathbb{E}[\widetilde{Y}_0 (s)])\leq f_2 (t,s,y,z,\mathbb{E}[\widetilde{Y}_0 (s)]), ~(t,y,z)\in [0,s]\times\mathbb{R}\times\mathbb{R},~a.s., a.e.~s\in[0,T];\\
\overline{\zeta} (t)\leq \zeta_2 (t),~a.s., a.e.~t\in[0,T].
\end{array}
\right.
\end{align*}
Applying Theorem 4.2 in \cite{Shi2020}, we have the following:
\begin{equation}\label{14}
\begin{aligned}
Y_2 (t)=\widetilde{Y}_0 (t) \geq\widetilde{Y}_1 (t),~a.s., a.e.~t\in[0,T].
\end{aligned}
\end{equation}
Then we consider the following equation:
\begin{equation}\label{15}
\begin{aligned}
\widetilde{Y}_2 (t)&=\overline{\zeta} (t)+\int_{t}^{T} \overline{f} (t, s, \widetilde{Y}_2 (s), \widetilde{Z}_2 (t,s), \mathbb{E}[\widetilde{Y}_1 (s)])ds\\
&~~~+\int_{t}^{T} g(t, s, \widetilde{Y}_2 (s), \widetilde{Z}_2 (t,s), \mathbb{E}[\widetilde{Y}_1 (s)]) d\overleftarrow{B}(s)-\int_{t}^{T} \widetilde{Z}_2 (t, s) d\overrightarrow{W}(s), ~~~~t\in[0, T].
\end{aligned}
\end{equation}
By Theorem 3.3 in \cite{Shi2020}, we know that there exists a unique solution $(\widetilde{Y}_2 (\cdot), \widetilde{Z}_2 (\cdot,\cdot))\in L_{\mathbb{F}}^{2} (0,T; \mathbb{R})\times L_{\mathbb{F}}^{2} (\Delta; \mathbb{R})$
to BDSVIE (\ref{15}).
Since $\overline{y}\mapsto \overline{f}(t,s,y,z,\overline{y})$ is nondecreasing, and using (\ref{14}), one can have that
\begin{align*}
\overline{f}(t,s,y,z,\mathbb{E}[\widetilde{Y}_0 (s)])\geq  \overline{f}(t,s,y,z,\mathbb{E}[\widetilde{Y}_1 (s)]), ~(t,y,z)\in [0,s]\times\mathbb{R}\times\mathbb{R},~a.s., a.e.~s\in[0,T].
\end{align*}
Following the similar approach, it is easy to see that
\begin{align*}
\widetilde{Y}_1 (t)\geq\widetilde{Y}_2 (t),~a.s., a.e.~t\in[0,T].
\end{align*}
Therefore, let the sequence $\{(\widetilde{Y}_p (\cdot), \widetilde{Z}_p (\cdot,\cdot))\}_{p=1}^{\infty}\in L_{\mathbb{F}}^{2} (0,T; \mathbb{R})\times L_{\mathbb{F}}^{2} (\Delta; \mathbb{R})$ satisfy the following equation:
\begin{equation}\label{17}
\begin{aligned}
\widetilde{Y}_p (t)&=\overline{\zeta} (t)+\int_{t}^{T} \overline{f} (t, s, \widetilde{Y}_p (s), \widetilde{Z}_p (t,s), \mathbb{E}[\widetilde{Y}_{p-1} (s)])ds\\
&~~~+\int_{t}^{T} g(t, s, \widetilde{Y}_p (s), \widetilde{Z}_p (t,s), \mathbb{E}[\widetilde{Y}_{p-1} (s)]) d\overleftarrow{B}(s)-\int_{t}^{T} \widetilde{Z}_p (t, s) d\overrightarrow{W}(s),~~t\in[0, T],
\end{aligned}
\end{equation}
we can deduce that
\begin{align*}
Y_2 (t)=\widetilde{Y}_0 (t) \geq\widetilde{Y}_1 (t)\geq\widetilde{Y}_2 (t)\geq\cdots\geq\widetilde{Y}_p (t)\geq\cdots,~a.s., a.e.~t\in[0,T].
\end{align*}

Next, we shall show that $\{(\widetilde{Y}_p (\cdot), \widetilde{Z}_p (\cdot,\cdot))\}_{p=1}^{\infty}$ is a Cauchy sequence.
Note that (2.2) in \cite{Shi2020} and by (\ref{16}), the following estimate holds:
\begin{equation*}
\begin{aligned}
&~~~~\mathbb{E}\left[\int_0^T e^{\beta t} |\widetilde{Y}_p (t)-\widetilde{Y}_q (t)|^2 dt+\int_0^T \int_t^T  e^{\beta s}|\widetilde{Z}_p (t,s)-\widetilde{Z}_q (t,s)|^2 dsdt\right]\\
&\leq\frac{10}{\beta} \mathbb{E}\left[\int_0^T \int_t^T e^{\beta s}\left|\overline{f}(t,s,\widetilde{Y}_p (s),\widetilde{Z}_p (t,s),\mathbb{E}[\widetilde{Y}_{p-1} (s)])\right.\right.\\
&~~~\left.\left.-\overline{f}(t,s,\widetilde{Y}_q (s),\widetilde{Z}_q (t,s),\mathbb{E}[\widetilde{Y}_{q-1} (s)])\right|^2 dsdt\right]\\
&~~~+\mathbb{E}\left[\int_0^T  \int_t^T e^{\beta s} \left| g(t,s,\widetilde{Y}_p (s),\widetilde{Z}_p (t,s),\mathbb{E}[\widetilde{Y}_{p-1} (s)])\right.\right.\\
&~~~\left.\left.-g (t,s,\widetilde{Y}_q (s),\widetilde{Z}_q (t,s),\mathbb{E}[\widetilde{Y}_{q-1} (s)])\right|^{2} dsdt\right]\\
&~~~+\mathbb{E}\left[\int_0^T e^{\beta t} \int_t^T  \left|g(t,s,\widetilde{Y}_p (s),\widetilde{Z}_p (t,s),\mathbb{E}[\widetilde{Y}_{p-1} (s)])\right.\right.\\
&~~~\left.\left.-g(t,s,\widetilde{Y}_q (s),\widetilde{Z}_q (t,s),\mathbb{E}[\widetilde{Y}_{q-1} (s)])\right|^{2} dsdt\right]\\
&\leq\gamma\mathbb{E}\left[\int_0^T e^{\beta t} |\widetilde{Y}_p (t)-\widetilde{Y}_q (t)|^2 dt
+\int_0^T e^{\beta t} |\widetilde{Y}_{p-1} (t)-\widetilde{Y}_{q-1} (t)|^2 dt\right.\\
&~~~\left.+\int_0^T \int_t^T e^{\beta s}|\widetilde{Z}_p (t,s)-\widetilde{Z}_q (t,s)|^2 dsdt\right],
\end{aligned}
\end{equation*}
where $\gamma=\frac{20c(T+1)+2\alpha}{\beta}+2\alpha(T+2)$. From $0<\alpha<\frac{1}{2(T+2)}$, and by choosing
$\beta=\frac{80c(T+1)+8\alpha}{1-6\alpha(T+2)}$, one can have the following:
\begin{equation*}
\begin{aligned}
&~~~~\mathbb{E}\left[\int_0^T e^{\beta t} |\widetilde{Y}_p (t)-\widetilde{Y}_q (t)|^2 dt+\int_0^T \int_t^T  e^{\beta s}|\widetilde{Z}_p (t,s)-\widetilde{Z}_q (t,s)|^2 dsdt\right]\\
&\leq\frac{\gamma}{1-\gamma}\mathbb{E}\left[\int_0^T e^{\beta t} |\widetilde{Y}_{p-1} (t)-\widetilde{Y}_{q-1} (t)|^2 dt\right],
\end{aligned}
\end{equation*}
where $\gamma=\frac{1+2\alpha(T+2)}{4}<\frac{1}{2}$. In other words, $\{(\widetilde{Y}_p (\cdot), \widetilde{Z}_p (\cdot,\cdot))\}_{p=1}^{\infty}$ is a Cauchy sequence in $L_{\mathbb{F}}^{2} (0$, $T;\mathbb{R})\times L_{\mathbb{F}}^{2} (\Delta; \mathbb{R})$ and converges to some $(\widetilde{Y} (\cdot), \widetilde{Z} (\cdot,\cdot))\in L_{\mathbb{F}}^{2} (0,T;\mathbb{R})\times L_{\mathbb{F}}^{2} (\Delta; \mathbb{R})$ such that
\begin{align*}
\lim_{p \rightarrow\infty} \mathbb{E}\left[\int_0^T e^{\beta t} |\widetilde{Y}_p (t)-\widetilde{Y} (t)|^2 dt+\int_0^T \int_t^T  e^{\beta s}|\widetilde{Z}_p (t,s)-\widetilde{Z} (t,s)|^2 dsdt\right]=0.
\end{align*}
Finally, we can let $p\rightarrow\infty$ in Eq. (\ref{17}) to obtain that
\begin{equation*}
\begin{aligned}
\widetilde{Y} (t)&=\overline{\zeta} (t)+\int_{t}^{T} \overline{f} (t, s, \widetilde{Y} (s), \widetilde{Z} (t,s), \mathbb{E}[\widetilde{Y} (s)])ds\\
&~~~+\int_{t}^{T} g(t, s, \widetilde{Y} (s), \widetilde{Z} (t,s), \mathbb{E}[\widetilde{Y} (s)]) d\overleftarrow{B}(s)-\int_{t}^{T} \widetilde{Z} (t, s) d\overrightarrow{W}(s), ~~~~t\in[0, T].
\end{aligned}
\end{equation*}
According to Theorem \ref{th1}, we derive that
\begin{align*}
Y_2 (t)=\widetilde{Y}_0 (t)\geq\widetilde{Y}(t)= \overline{Y}(t),~a.s., a.e.~t\in[0,T].
\end{align*}
Similarly, one can show that
\begin{align*}
\overline{Y} (t)\geq Y_1 (t),~a.s., a.e.~t\in[0,T].
\end{align*}
The proof is therefore complete.
\end{proof}

\section{Regularity result for the M-solution of MF-BDSVIEs}\label{sec5}

The discussion in this section is aimed to prove the differentiability
of the M-solution to MF-BDSVIE (\ref{3}) through Malliavin calculus.

Denote by $\Phi^2 [0, T]$ the space of all processes $\phi(\cdot)\in L^{\infty} ([0,T]; L_{\mathcal{F}_T}^{2} (\Omega; \mathbb{R}^n))$ such that
\begin{align*}
\|\phi(\cdot)\|_{\Phi^2 [0, T]}^2
:=\sup\limits_{t\in [0, T]} \mathbb{E}\left[ |\phi(t)|^2 +\int_0^T \sum\limits_{i=1}^{m} |D_t^i \phi(s)|^2 ds\right].
\end{align*}

For convenience, MF-BDSVIE (\ref{3}) can be rewritten as
\begin{equation}\label{45}
\begin{aligned}
Y(t)&=\zeta(t)+\int_{t}^{T} f(t, s, Y(s), Z(t, s), Z(s, t), \mathbb{E}[Y(s)], \mathbb{E}[Z(t, s)], \mathbb{E}[Z(s, t)])ds\\
&~~~+\int_{t}^{T} g(t, s, Y(s), Z(t, s), Z(s, t), \mathbb{E}[Y(s)], \mathbb{E}[Z(t, s)], \mathbb{E}[Z(s, t)]) d\overleftarrow{B}(s)\\
&~~~-\int_{t}^{T} Z(t, s) d\overrightarrow{W}(s), ~~~~t\in[0, T],
\end{aligned}
\end{equation}
where $f: \Omega\times \Delta \times (\mathbb{R}^n \times\mathbb{R}^{n\times m}\times\mathbb{R}^{n\times m})^2
\mapsto\mathbb{R}^n$ and
$g: \Omega\times \Delta \times (\mathbb{R}^n \times\mathbb{R}^{n\times m}\times\mathbb{R}^{n\times m})^2
\mapsto \mathbb{R}^{n\times k}$ are jointly measurable maps, and satisfy the assumptions:
\begin{itemize}
\item[(A4)]
Let $(y,z,\varsigma,\overline{y},\overline{z},\overline{\varsigma})\mapsto
f(t,s,y,z,\varsigma,\overline{y},\overline{z},\overline{\varsigma})$ and $(y,z,\varsigma,\overline{y},\overline{z},\overline{\varsigma})\mapsto
g(t,s,y,z,\varsigma,\overline{y},\overline{z},\overline{\varsigma})$ be continuously differentiable.
There exist constants $c>0$ and $0<\alpha<\frac{1}{2(T+2)}$
such that the norm of partial derivatives of $f$
is bounded by $c$,
and the norm of partial derivatives of $g$ is bounded by $\alpha$.
\item[(A5)]
For each $1\leq i \leq m$, let $(y,z,\varsigma,\overline{y},\overline{z},\overline{\varsigma})\mapsto
\left[D_r^i f\right](t,s,y,z,\varsigma,\overline{y},\overline{z},\overline{\varsigma})$ and $(y,z,\varsigma,\overline{y},\overline{z},\overline{\varsigma})\mapsto
\left[D_r^i g\right](t,s,y,z,\varsigma,\overline{y}$, $\overline{z},\overline{\varsigma})$ be continuous for all $r\in [0,T]$.
There exist stochastic processes $L_1 (t,s), L_2 (t,s): \Omega\times\Delta\mapsto [0,+\infty)$ satisfying
\begin{align*}
\mathbb{E}\left[\int_0^T\int_t^T |L_1 (t,s)|^2 dsdt \right]<\infty,~~~\mathbb{E}\left[\int_0^T\int_t^T |L_2 (t,s)|^2 dsdt \right]<\frac{1}{2(T+2)}<\infty,
\end{align*}
such that for any $(r,t,s,y,z,\varsigma,\overline{y},\overline{z},\overline{\varsigma})
\in [0,T]\times\Delta\times(\mathbb{R}^n \times\mathbb{R}^{n\times m}\times\mathbb{R}^{n\times m})^2$, we have
\begin{align*}
\sum\limits_{i=1}^{m} \left|\left[D_r^i f\right](t,s,y,z,\varsigma,\overline{y},\overline{z},\overline{\varsigma})\right|\leq L_1 (t,s),\\
\sum\limits_{i=1}^{m} \left\|\left[D_r^i g\right](t,s,y,z,\varsigma,\overline{y},\overline{z},\overline{\varsigma})\right\|\leq L_2 (t,s).
\end{align*}
\end{itemize}

We are now in a position to prove the regularity of the M-solution to MF-BDSVIE (\ref{45}) as the main theorem of this section.
\begin{theorem}\label{th6}
Under $(A4)$-$(A5)$, for any $\zeta (\cdot)\in \Phi^2 [0, T]$, let $(Y(\cdot), Z(\cdot,\cdot))\in \mathcal{M}_\beta^2 [0,T]~(\beta>0)$ be the unique M-solution of MF-BDSVIE (\ref{45}).
Then for any $0\leq r \leq T$ and $1\leq i\leq m$, $(D_r^i Y(\cdot), D_r^i Z(\cdot,\cdot))$ is the unique M-solution of the following MF-BDSVIE:
\begin{equation}\label{33}
\begin{aligned}
D_r^i Y(t)&=D_r^i \zeta (t)+\int_t^T \left\{\left[D_r^i f\right]\left(\Gamma (t,s)\right)+f_y \left(\Gamma (t,s)\right)D_r^i Y(s)
+f_{\bar{y}} \left(\Gamma (t,s)\right)\mathbb{E}[D_r^i Y(s)]\right.\\
&~~+\sum\limits_{j=1}^{m} \left( f_{z_j}\left(\Gamma (t,s)\right)D_r^i Z_j (t,s)+f_{\varsigma_j}\left(\Gamma (t,s)\right)D_r^i Z_j (s,t)
+f_{\bar{z}_j}\left(\Gamma (t,s)\right)\mathbb{E}[D_r^i Z_j (t,s)]\right.\\
&~~+\left.\left.f_{\bar{\varsigma}_j}\left(\Gamma (t,s)\right)\mathbb{E}[D_r^i Z_j (s,t)]\right)\right\} ds\\
&~~+\sum\limits_{l=1}^{k}\int_t^T\left\{ \left[D_r^i g^l \right]\left(\Gamma (t,s)\right)
+g_y^l \left(\Gamma (t,s)\right)D_r^i Y(s)
+g_{\bar{y}}^l \left(\Gamma (t,s)\right)\mathbb{E}[D_r^i Y(s)]\right.\\
&~~+\sum\limits_{j=1}^{m} \left( g_{z_j}^l \left(\Gamma (t,s)\right)D_r^i Z_j (t,s)+g_{\varsigma_j}^l \left(\Gamma (t,s)\right)D_r^i Z_j (s,t)
+g_{\bar{z}_j}^l \left(\Gamma (t,s)\right)\mathbb{E}[D_r^i Z_j (t,s)]\right.\\
&~~+\left.\left.g_{\bar{\varsigma}_j}^l \left(\Gamma (t,s)\right)\mathbb{E}[D_r^i Z_j (s,t)]\right)\right\} d\overleftarrow{B}_l (s) -\int_t^T D_r^i Z (t,s) d\overrightarrow{W}(s),~~~(t,r)\in \Delta^c,
\end{aligned}
\end{equation}
where
\begin{align*}
\Gamma (t,s):= (t,s,Y(s),Z(t,s),Z(s,t),\mathbb{E}[Y(s)],\mathbb{E}[Z(t,s)],\mathbb{E}[Z(s,t)]),
\end{align*}
$\{Z_j (\cdot,\cdot)\}_{1\leq j \leq m}$ is the $j$-th column of the matrix $Z(\cdot,\cdot)$,
$\{g^l (\cdot)\}_{1\leq l \leq k}$ is the $l$-th column of the matrix $g(\cdot)$, and
$\overleftarrow{B}(\cdot)=(\overleftarrow{B}_1(\cdot),\ldots,\overleftarrow{B}_k(\cdot))$.\\
Furthermore,
\begin{equation}\label{41}
\begin{aligned}
D_r^i Y(t)=Z_i (t,r)+\int_r^t D_r^i Z (t,s) d\overrightarrow{W}(s),~~~(t,r)\in \Delta^c,
\end{aligned}
\end{equation}
which means that
\begin{equation}\label{42}
\begin{aligned}
Z_i (t,r)=\mathbb{E}\left[D_r^i Y(t)~|~\mathcal{F}_r \right],~~~(t,r)\in \Delta^c,~a.s.
\end{aligned}
\end{equation}
Otherwise,
\begin{equation}\label{34}
\begin{aligned}
Z_i (t,r)&=D_r^i \zeta (t)+\int_r^T \left\{\left[D_r^i f\right]\left(\Gamma (t,s)\right)+f_y \left(\Gamma (t,s)\right)D_r^i Y(s)
+f_{\bar{y}} \left(\Gamma (t,s)\right)\mathbb{E}[D_r^i Y(s)]\right.\\
&~~+\left.\sum\limits_{j=1}^{m} \left( f_{z_j}\left(\Gamma (t,s)\right)D_r^i Z_j (t,s)
+f_{\bar{z}_j}\left(\Gamma (t,s)\right)\mathbb{E}[D_r^i Z_j (t,s)]\right)\right\} ds\\
&~~+\sum\limits_{l=1}^{k}\int_r^T\left\{ \left[D_r^i g^l \right]\left(\Gamma (t,s)\right)
+g_y^l \left(\Gamma (t,s)\right)D_r^i Y(s)
+g_{\bar{y}}^l \left(\Gamma (t,s)\right)\mathbb{E}[D_r^i Y(s)]\right.\\
&~~+\left.\sum\limits_{j=1}^{m} \left( g_{z_j}^l \left(\Gamma (t,s)\right)D_r^i Z_j (t,s)
+g_{\bar{z}_j}^l \left(\Gamma (t,s)\right)\mathbb{E}[D_r^i Z_j (t,s)]
\right)\right\} d\overleftarrow{B}_l (s) \\
&~~-\int_r^T D_r^i Z (t,s) d\overrightarrow{W}(s),~~~(t,r)\in \Delta.
\end{aligned}
\end{equation}
\end{theorem}
\begin{proof}
For convenience, this proof only considers the one-dimensional case.
According to Theorem \ref{th1}, if $c>0$, $0<\alpha<\frac{1}{2(T+2)}$
and $\beta>\frac{20c(T+1)+2\alpha}{1-2\alpha(T+2)}$,
then $\Gamma$ defined by (\ref{31}) is a contractive map on the space $\mathcal{M}_\beta^2 [0,T]$.
Set $(Y_0 (\cdot), Z_0 (\cdot,\cdot))=0$, and for $d=0,1,2,\cdots$, we define the Picard iterations
\begin{align*}
(Y_{d+1} (\cdot), Z_{d+1} (\cdot,\cdot))=\Gamma (Y_{d} (\cdot), Z_{d} (\cdot,\cdot)).
\end{align*}
Obviously, $\{(Y_{d} (\cdot), Z_{d} (\cdot,\cdot))\}_{d=0}^{\infty}$ converges to the unique M-solution
$(Y (\cdot), Z (\cdot,\cdot))$ of MF-BDSVIE (\ref{45}), which implies that
\begin{equation*} \begin{aligned}
\lim\limits_{d\rightarrow \infty}\|(Y_d (\cdot),Z_d (\cdot,\cdot))-(Y(\cdot),Z(\cdot,\cdot))\|_{\mathcal{M}_\beta^2 [0,T]}=0.
\end{aligned}
\end{equation*}
Hereafter, we consider the following MF-BDSVIE:
\begin{equation*}
\begin{aligned}
Y_{d+1}(t)&=\zeta(t)+\int_{t}^{T} f(t, s, Y_d (s), Z_d (t, s), Z_d (s, t), \mathbb{E}[Y_d (s)], \mathbb{E}[Z_d (t, s)], \mathbb{E}[Z_d (s, t)])ds
\end{aligned}
\end{equation*}
\begin{equation*}
\begin{aligned}
&~~~+\int_{t}^{T} g(t, s, Y_d (s), Z_d (t, s), Z_d (s, t), \mathbb{E}[Y_d (s)], \mathbb{E}[Z_d (t, s)], \mathbb{E}[Z_d (s, t)]) d\overleftarrow{B}(s)\\
&~~~-\int_{t}^{T} Z_{d+1}(t, s) d\overrightarrow{W}(s), ~~~~t\in[0, T].
\end{aligned}
\end{equation*}
It then follows from Proposition 5.3 in \cite{Karoui1997} that
\begin{align*}
(D_r Y_d (\cdot),D_r Z_d (\cdot,\cdot))\in\mathcal{M}_\beta^2 [0,T],~~d=0,1,2,\cdots,
\end{align*}
and owing to Lemma 5.1 in \cite{Karoui1997} and Lemma 4.2 in \cite{Bachouch2016}, we can have
\begin{equation}\label{35}
\begin{aligned}
D_r Y_{d+1} (t)&=D_r \zeta (t)+\int_t^T \left\{\left[D_r f\right]\left(\Gamma_d (t,s)\right)+f_y \left(\Gamma_d (t,s)\right)D_r Y_d (s)+ f_{z}\left(\Gamma_d (t,s)\right)D_r Z_d (t,s)
\right.\\
&~~+f_{\varsigma}\left(\Gamma_d (t,s)\right)D_r Z_d (s,t)
+f_{\bar{y}} \left(\Gamma_d (t,s)\right)\mathbb{E}[D_r Y_d (s)]
+f_{\bar{z}}\left(\Gamma_d (t,s)\right)\mathbb{E}[D_r Z_d (t,s)]\\
&~~\left.+f_{\bar{\varsigma}}\left(\Gamma_d (t,s)\right)\mathbb{E}[D_r Z_d (s,t)]\right\} ds\\
&~~+\int_t^T\left\{ \left[D_r g \right]\left(\Gamma_d (t,s)\right)
+g_y \left(\Gamma_d (t,s)\right)D_r Y_d (s)+g_{z} \left(\Gamma_d (t,s)\right)D_r Z_d (t,s)
\right.\\
&~~+g_{\varsigma} \left(\Gamma_d (t,s)\right)D_r Z_d (s,t)+g_{\bar{y}} \left(\Gamma_d (t,s)\right)\mathbb{E}[D_r Y_d (s)]
+g_{\bar{z}} \left(\Gamma_d (t,s)\right)\mathbb{E}[D_r Z_d (t,s)]\\
&~~\left.+g_{\bar{\varsigma}} \left(\Gamma_d (t,s)\right)\mathbb{E}[D_r Z_d (s,t)]\right\} d\overleftarrow{B}_l (s) -\int_t^T D_r Z_{d+1} (t,s) d\overrightarrow{W}(s),\\
&~~(t,r)\in \Delta^c,~d=0,1,2,\cdots,
\end{aligned}
\end{equation}
where
\begin{align*}
\Gamma_d (t,s):= (t,s,Y_d (s),Z_d (t,s),Z_d (s,t),\mathbb{E}[Y_d (s)],\mathbb{E}[Z_d (t,s)],\mathbb{E}[Z_d (s,t)]).
\end{align*}
By Theorem \ref{th1}, let $(\tilde{Y}^{r} (\cdot), \tilde{Z}^{r} (\cdot,\cdot))\in\mathcal{M}_\beta^2 [0,T]$ be the unique M-solution of the linear MF-BDSVIE:
\begin{equation}\label{36}
\begin{aligned}
\tilde{Y}^r (t) &=D_r \zeta (t)+\int_t^T \left\{\left[D_r f\right]\left(\Gamma (t,s)\right)+f_y \left(\Gamma (t,s)\right) \tilde{Y}^r (s)+ f_{z}\left(\Gamma (t,s)\right) \tilde{Z}^r (t,s)+f_{\varsigma}\left(\Gamma (t,s)\right)
\right.\\
&~~\left.\cdot\tilde{Z}^r (s,t)
+f_{\bar{y}} \left(\Gamma (t,s)\right)\mathbb{E}[\tilde{Y}^r (s)]+f_{\bar{z}}\left(\Gamma (t,s)\right)\mathbb{E}[\tilde{Z}^r (t,s)]
+f_{\bar{\varsigma}}\left(\Gamma (t,s)\right)\mathbb{E}[\tilde{Z}^r (s,t)]\right\} ds\\
&~~+\int_t^T\left\{ \left[D_r g \right]\left(\Gamma (t,s)\right)
+g_y \left(\Gamma (t,s)\right)\tilde{Y}^r (s)+ g_{z} \left(\Gamma (t,s)\right)\tilde{Z}^r (t,s)+g_{\varsigma} \left(\Gamma (t,s)\right)\tilde{Z}^r (s,t)
\right.\\
&~~\left.
+g_{\bar{y}} \left(\Gamma (t,s)\right)\mathbb{E}[\tilde{Y}^r (s)]+g_{\bar{z}} \left(\Gamma (t,s)\right)\mathbb{E}[\tilde{Z}^r (t,s)]+g_{\bar{\varsigma}} \left(\Gamma (t,s)\right)\mathbb{E}[\tilde{Z}^r (s,t)]\right\} d\overleftarrow{B} (s)\\
&~~-\int_t^T \tilde{Z}^r (t,s) d\overrightarrow{W}(s),
~~(t,r)\in \Delta^c.
\end{aligned}
\end{equation}
Using Eqs. (\ref{35})-(\ref{36}), and by the stability estimate (\ref{32}) in Theorem \ref{th1}, one can deduces that
\begin{equation}\label{39}
\begin{aligned}
&~~~\mathbb{E}\left[\int_0^T e^{\beta t} |D_r Y_{d+1} (t)-\tilde{Y}^r (t)|^2 dt+\int_0^T \int_t^T  e^{\beta s} |D_r Z_{d+1} (t,s)-\tilde{Z}^r (t,s)|^2 dsdt\right]
\\
&\leq L \mathbb{E}\left[
\int_0^T \int_t^T  e^{\beta s} |\delta f^d_r (t,s)|^2 dsdt
+\int_0^T \int_t^T  e^{\beta s} |\delta g^d_r (t,s)|^2 dsdt \right],
\end{aligned}
\end{equation}
where
\begin{equation*}
\begin{aligned}
\delta f^d_r (t,s)=&\left[D_r f\right]\left(\Gamma_d (t,s)\right)-\left[D_r f\right]\left(\Gamma (t,s)\right)
+\left(f_y \left(\Gamma_d (t,s)\right)-f_y \left(\Gamma (t,s)\right)\right) \tilde{Y}^r (s)\\
&+\left(f_{z}\left(\Gamma_d (t,s)\right)-f_{z}\left(\Gamma (t,s)\right)\right) \tilde{Z}^r (t,s)
+\left(f_{\varsigma}\left(\Gamma_d (t,s)\right)-f_{\varsigma}\left(\Gamma (t,s)\right)\right) \tilde{Z}^r (s,t)\\
&+\left(f_{\bar{y}} \left(\Gamma_d (t,s)\right)-f_{\bar{y}} \left(\Gamma (t,s)\right)\right)\mathbb{E}[\tilde{Y}^r (s)]
+\left(f_{\bar{z}} \left(\Gamma_d (t,s)\right)-f_{\bar{z}}\left(\Gamma (t,s)\right)\right)\mathbb{E}[\tilde{Z}^r (t,s)] \\ &+\left(f_{\bar{\varsigma}}\left(\Gamma_d (t,s)\right)-f_{\bar{\varsigma}}\left(\Gamma (t,s)\right)\right)\mathbb{E}[\tilde{Z}^r (s,t)],\\
\delta g^d_r (t,s)=&\left[D_r g\right]\left(\Gamma_d (t,s)\right)-\left[D_r g\right]\left(\Gamma (t,s)\right)
+\left(g_y \left(\Gamma_d (t,s)\right)-g_y \left(\Gamma (t,s)\right)\right) \tilde{Y}^r (s)\\
&+\left(g_{z}\left(\Gamma_d (t,s)\right)-g_{z}\left(\Gamma (t,s)\right)\right) \tilde{Z}^r (t,s)
+\left(g_{\varsigma}\left(\Gamma_d (t,s)\right)-g_{\varsigma}\left(\Gamma (t,s)\right)\right) \tilde{Z}^r (s,t)\\
&+\left(g_{\bar{y}} \left(\Gamma_d (t,s)\right)-g_{\bar{y}} \left(\Gamma (t,s)\right)\right)\mathbb{E}[\tilde{Y}^r (s)]
+\left(g_{\bar{z}} \left(\Gamma_d (t,s)\right)-g_{\bar{z}}\left(\Gamma (t,s)\right)\right)\mathbb{E}[\tilde{Z}^r (t,s)] \\ &+\left(g_{\bar{\varsigma}}\left(\Gamma_d (t,s)\right)-g_{\bar{\varsigma}}\left(\Gamma (t,s)\right)\right)\mathbb{E}[\tilde{Z}^r (s,t)].
\end{aligned}
\end{equation*}
From the Jensen inequality one has
\begin{equation}\label{43}
\begin{aligned}
&~~~\mathbb{E}\left[
\int_0^T \int_t^T  e^{\beta s} |\delta f^d_r (t,s)|^2 dsdt \right]\\
&\leq L \mathbb{E}\left[\int_0^T \int_t^T  e^{\beta s} \left|\left[D_r f\right]\left(\Gamma_d (t,s)\right)-\left[D_r f\right]\left(\Gamma (t,s)\right) \right|^2 dsdt \right.\\
&+\int_0^T \int_t^T  e^{\beta s} \left(\left| f_y \left(\Gamma_d (t,s)\right)-f_y \left(\Gamma (t,s)\right)\right|^2
+ \mathbb{E}\left[\left| f_{\bar{y}} \left(\Gamma_d (t,s)\right)-f_{\bar{y}} \left(\Gamma (t,s)\right)\right|^2\right]  \right) \left|\tilde{Y}^r (s) \right|^2 dsdt \\
&+\int_0^T \int_t^T  e^{\beta s} \left(\left| f_{z}\left(\Gamma_d (t,s)\right)-f_{z}\left(\Gamma (t,s)\right)\right|^2
+\mathbb{E}\left[\left| f_{\bar{z}}\left(\Gamma_d (t,s)\right)-f_{\bar{z}}\left(\Gamma (t,s)\right)\right|^2\right]\right) \left|\tilde{Z}^r (t,s) \right|^2 dsdt \\
&+\left.\int_0^T \int_t^T  e^{\beta s} \left(\left|f_{\varsigma}\left(\Gamma_d (t,s)\right)-f_{\varsigma}\left(\Gamma (t,s)\right)\right|^2
+\mathbb{E}\left[\left|f_{\bar{\varsigma}}\left(\Gamma_d (t,s)\right)-f_{\bar{\varsigma}}\left(\Gamma (t,s)\right)\right|^2\right]\right)
\left|\tilde{Z}^r (s,t) \right|^2 dsdt \right].
\end{aligned}
\end{equation}
Because $D_r f$, $f_y$, $f_z$, $f_{\varsigma}$, $f_{\bar{y}}$, $f_{\bar{z}}$, $f_{\bar{\varsigma}}$ are bounded and continuous
with respect to $(y,z,\varsigma,\bar{y},\bar{z},\bar{\varsigma})$, using the dominated convergence theorem and (\ref{30}),
we can let $d\rightarrow\infty$ in (\ref{43}) to show that
\begin{equation}\label{37}
\begin{aligned}
\lim\limits_{d\rightarrow\infty}\mathbb{E}\left[
\int_0^T \int_t^T  e^{\beta s} |\delta f^d_r (t,s)|^2 dsdt \right]=0.
\end{aligned}
\end{equation}
In the same way as in the proof of (\ref{37}) yields that
\begin{equation}\label{38}
\begin{aligned}
\lim\limits_{d\rightarrow\infty}\mathbb{E}\left[
\int_0^T \int_t^T  e^{\beta s} |\delta g^d_r (t,s)|^2 dsdt \right]=0.
\end{aligned}
\end{equation}
In view of (\ref{37})-(\ref{38}), and letting $d\rightarrow\infty$ in (\ref{39}), we obtain that
\begin{equation*}
\begin{aligned}
\lim\limits_{d\rightarrow\infty}\mathbb{E}\left[\int_0^T e^{\beta t} |D_r Y_{d+1} (t)-\tilde{Y}^r (t)|^2 dt+\int_0^T \int_t^T  e^{\beta s} |D_r Z_{d+1} (t,s)-\tilde{Z}^r (t,s)|^2 dsdt\right]=0.
\end{aligned}
\end{equation*}
Hence, due to the fact that the operator $D_r$ is closed, we know that
\begin{equation*}
\begin{aligned}
\tilde{Y}^r (t)=D_r Y (t),~~\tilde{Z}^r (t,s)=D_r Z(t,s),~~t,s\in [0,T], a.s.,
\end{aligned}
\end{equation*}
and then Eq. (\ref{33}) holds. By the definition of M-solutions deduced that
\begin{align*}
Y(t)=\mathbb{E}[Y(t)]+\int_0^t Z(t,s) d\overrightarrow{W}(s),~~~a.e.~t\in [0,T],
\end{align*}
and (\ref{41})-(\ref{42}) derived Lemma 5.1 in \cite{Karoui1997}.
Clearly, Eq. (\ref{34}) follows from Lemma 5.1 in \cite{Karoui1997} and Lemma 4.2 in \cite{Bachouch2016}.
\end{proof}

\section{Dynamic risk measure}\label{sec6}
As an application of MF-BDSVIEs to risk management,
this section is devoted to the dynamic risk measure by means of MF-BDSVIEs.

First, we review the definition of dynamic risk measures shown in \cite{Yong2007}.
\begin{definition}\label{def2}
(See \cite{Yong2007}) A map $\rho: [0,T]\times L_{\mathcal{F}_T}^{2} (0,T; \mathbb{R})\mapsto L_{\mathbb{F}}^{2} (0,T; \mathbb{R})$ is called a dynamic risk measure if the following conditions hold:\\
$(1)$ (Past independence) ~Take each $\zeta(\cdot), \bar{\zeta}(\cdot)\in L_{\mathcal{F}_T}^{2} (0,T; \mathbb{R})$, if
\begin{align*}
\zeta(s)=\bar{\zeta}(s),~a.s., ~s\in[t,T],
\end{align*}
for some $t\in[0,T]$, then
\begin{align*}
\rho(t;\zeta(\cdot))=\rho(t;\bar{\zeta}(\cdot)), ~a.s.
\end{align*}
$(2)$ (Monotonicity)~Take each $\zeta(\cdot), \bar{\zeta}(\cdot)\in L_{\mathcal{F}_T}^{2} (0,T; \mathbb{R})$, if
\begin{align*}
\zeta(s)\leq\bar{\zeta}(s),~a.s., ~s\in[t,T],
\end{align*}
for some $t\in[0,T]$, then
\begin{align*}
\rho(s;\zeta(\cdot))\geq\rho(s;\bar{\zeta}(\cdot)), ~a.s., ~s\in[t,T].
\end{align*}
In addition, the dynamic risk measure $\rho: [0,T]\times L_{\mathcal{F}_T}^{2} (0,T; \mathbb{R})\mapsto L_{\mathbb{F}}^{2} (0,T; \mathbb{R})$ is called a dynamic convex risk measure if the following conditions hold:\\
$(3)$ (Translation invariance)~There is a deterministic integrable function $r(\cdot)$ such that for each $\zeta(\cdot)\in L_{\mathcal{F}_T}^{2} (0,T; \mathbb{R})$ and $c\in\mathbb{R}$,
\begin{align*}
\rho(t;\zeta(\cdot)+c)=\rho(t;\zeta(\cdot))-c e^{-\int_t^T r(s)ds}, ~a.s., ~t\in[0,T].
\end{align*}
$(4)$ (Convexity)~Take each $\zeta(\cdot), \bar{\zeta}(\cdot)\in L_{\mathcal{F}_T}^{2} (0,T; \mathbb{R})$ and $\lambda\in [0,1]$,
\begin{align*}
\rho(t;\lambda\zeta(\cdot)+(1-\lambda)\bar{\zeta}(\cdot))\leq\lambda\rho(t;\zeta(\cdot))+(1-\lambda)\rho(t;\bar{\zeta}(\cdot)), ~a.s., ~t\in[0,T].
\end{align*}
Meanwhile, the dynamic risk measure $\rho: [0,T]\times L_{\mathcal{F}_T}^{2} (0,T; \mathbb{R})\mapsto L_{\mathbb{F}}^{2} (0,T; \mathbb{R})$ is called a dynamic coherent risk measure if the axiom $(4)$ is substituted by the following conditions:\\
$(5)$ (Positive homogeneity)~Take each $\zeta(\cdot)\in L_{\mathcal{F}_T}^{2} (0,T; \mathbb{R})$ and $\lambda>0$,
\begin{align*}
\rho(t;\lambda\zeta(\cdot))=\lambda\rho(t;\zeta(\cdot)), ~a.s., ~t\in[0,T].
\end{align*}
$(6)$ (Subadditivity)~Take each $\zeta(\cdot), \bar{\zeta}(\cdot)\in L_{\mathcal{F}_T}^{2} (0,T; \mathbb{R})$,
\begin{align*}
\rho(t;\zeta(\cdot)+\bar{\zeta}(\cdot))\leq\rho(t;\zeta(\cdot))+\rho(t;\bar{\zeta}(\cdot)), ~a.s., ~t\in[0,T].
\end{align*}
\end{definition}

Next, for any $\zeta(\cdot)\in L_{\mathcal{F}_T}^{2} (0,T; \mathbb{R})$, we define
\begin{align}\label{47}
\rho(t;\zeta(\cdot)):=Y^{-\zeta(\cdot)} (t), ~t\in[0,T],
\end{align}
where $Y^{-\zeta(\cdot)}(\cdot)$ is the $1$-st component of the unique solution $(Y^{-\zeta(\cdot)}(\cdot), Z^{-\zeta(\cdot)}(\cdot,\cdot))$
to the MF-BDSVIE:
\begin{equation}\label{44}
\begin{aligned}
Y^{-\zeta(\cdot)} (t)&=-\zeta (t)+\int_{t}^{T} f (t, s, Y^{-\zeta(\cdot)} (s), Z^{-\zeta(\cdot)} (t, s), \mathbb{E}[Y^{-\zeta(\cdot)} (s)])ds\\
&~~~+\int_{t}^{T} g(t, s, Z^{-\zeta(\cdot)} (t, s)) d\overleftarrow{B}(s)
-\int_{t}^{T} Z^{-\zeta(\cdot)} (t, s) d\overrightarrow{W}(s), ~~t\in[0, T],
\end{aligned}
\end{equation}
here $f: \Omega\times \Delta \times \mathbb{R} \times\mathbb{R}\times\mathbb{R}
\mapsto\mathbb{R}$ and $g: \Omega\times \Delta \times \mathbb{R}
\mapsto\mathbb{R}$ are given jointly measurable maps satisfying the assumptions $(A1)$-$(A3)$.\\
Denote by
\begin{align}\label{46}
f(t,s,y,z,\bar{y})=-\frac{r(s)}{2} (y+\bar{y})+h(t,s,z),
\end{align}
where $r(\cdot)$ is a deterministic integral function, and $h(t,s,z): \Omega\times \Delta \times \mathbb{R}
\mapsto\mathbb{R}$ is a given jointly measurable map.

In the sequel, we deal with the properties of dynamic risk measures induced by MF-BDSVIE (\ref{44}).
\begin{theorem}\label{th7}
Suppose that $f$ is given by (\ref{46}), and $f$ and $g$ satisfy $(A1)$-$(A3)$. Then $\rho$ defined by (\ref{47})
is a dynamic convex risk measure if the following conditions hold:\\
$(i)$ (Past independence) ~Take each $\zeta_1(\cdot), \zeta_2(\cdot)\in L_{\mathcal{F}_T}^{2} (0,T; \mathbb{R})$, if
\begin{align*}
\zeta_1(s)=\zeta_2(s),~a.s., ~s\in[t,T],
\end{align*}
for some $t\in[0,T]$, then
\begin{align*}
\rho(t;\zeta_1(\cdot))=\rho(t;\zeta_2(\cdot)), ~a.s.
\end{align*}
$(ii)$ (Monotonicity)~Take each $\zeta_1(\cdot), \zeta_2(\cdot)\in L_{\mathcal{F}_T}^{2} (0,T; \mathbb{R})$, if
\begin{align*}
\zeta_1(s)\leq\zeta_2(s),~a.s., ~s\in[t,T],
\end{align*}
for some $t\in[0,T]$, then
\begin{align*}
\rho(s;\zeta_1(\cdot))\geq\rho(s;\zeta_2(\cdot)), ~a.s., ~s\in[t,T].
\end{align*}
$(iii)$ (Translation invariance)~If $r(\cdot)$ is a deterministic integrable function, then for each $\zeta(\cdot)\in L_{\mathcal{F}_T}^{2} (0,T; \mathbb{R})$ and $c\in\mathbb{R}$,
\begin{align*}
\rho(t;\zeta(\cdot)+c)=\rho(t;\zeta(\cdot))-c e^{-\int_t^T r(s)ds}, ~a.s., ~t\in[0,T].
\end{align*}
$(iv)$ (Convexity)~If $f$ is convex with respect to $(y,z,\bar{y})$, that is,
\begin{equation}\label{50}
\begin{aligned}
&~~~~f(t,s,\lambda y_1+(1-\lambda)y_2,\lambda z_1+(1-\lambda)z_2,\lambda \bar{y}_1+(1-\lambda)\bar{y}_2)\\
&\leq\lambda f(t,s,y_1,z_1,\bar{y}_1)+(1-\lambda) f(t,s,y_2,z_2,\bar{y}_2), ~~a.s., ~(t,s)\in\Delta,~\lambda\in[0,1],
\end{aligned}
\end{equation}
and $g$ satisfies the following condition:
\begin{align}\label{51}
g(t,s,\lambda z_1+(1-\lambda)z_2)
=\lambda g(t,s,z_1)+(1-\lambda) g(t,s,z_2), ~a.s., ~(t,s)\in\Delta,~\lambda\in[0,1],
\end{align}
then take each $\zeta_1(\cdot), \zeta_2(\cdot)\in L_{\mathcal{F}_T}^{2} (0,T; \mathbb{R})$ and $\lambda\in [0,1]$,
\begin{align}\label{52}
\rho(t;\lambda\zeta_1(\cdot)+(1-\lambda)\zeta_2(\cdot))\leq\lambda\rho(t;\zeta_1(\cdot))+(1-\lambda)\rho(t;\zeta_2(\cdot)), ~a.s., ~t\in[0,T].
\end{align}
Moreover, suppose that $f$ is given by (\ref{46}), and $f$ and $g$ satisfy $(A1)$-$(A3)$. Then $\rho$ defined by (\ref{47})
is a dynamic coherent risk measure if $(i)$-$(iii)$ and the following conditions hold:\\
$(v)$ (Positive homogeneity)~If $f$ is positive homogeneity with respect to $(y,z,\bar{y})$, and
$g$ is positive homogeneity with respect to $z$, that is,
\begin{align*}
&f(t,s,\lambda y,\lambda z,\lambda\bar{y})
=\lambda f(t,s,y,z,\bar{y}), ~~a.s., ~(t,s)\in\Delta,~\lambda > 0,\\
&~~~~~~~g(t,s,\lambda z)
=\lambda g(t,s,z), ~~a.s., ~(t,s)\in\Delta,~\lambda >0,
\end{align*}
then take each $\zeta(\cdot)\in L_{\mathcal{F}_T}^{2} (0,T; \mathbb{R})$ and $\lambda>0$,
\begin{align*}
\rho(t;\lambda\zeta(\cdot))=\lambda\rho(t;\zeta(\cdot)), ~a.s., ~t\in[0,T].
\end{align*}
$(vi)$ (Subadditivity)~If $f$ is subadditivity with respect to $(y,z,\bar{y})$, that is,
\begin{align*}
f(t,s,y_1+y_2, z_1+z_2,\bar{y}_1+\bar{y}_2)
\leq f(t,s,y_1,z_1,\bar{y}_1)+ f(t,s,y_2,z_2,\bar{y}_2), ~a.s., ~(t,s)\in\Delta,
\end{align*}
and $g$ is additivity with respect to $z$, that is,
\begin{align*}
g(t,s, z_1+z_2)
= g(t,s,z_1)+ g(t,s,z_2), ~a.s., ~(t,s)\in\Delta,
\end{align*}
then take each $\zeta_1(\cdot), \zeta_2(\cdot)\in L_{\mathcal{F}_T}^{2} (0,T; \mathbb{R})$,
\begin{align*}
\rho(t;\zeta_1(\cdot)+\zeta_2(\cdot))\leq\rho(t;\zeta_1(\cdot))+\rho(t;\zeta_2(\cdot)), ~a.s., ~t\in[0,T].
\end{align*}
\end{theorem}
\begin{proof}
$(i)$ By the definition of $\rho$, it is easy to get the past independence.\\
$(ii)$ Since $\zeta_1(\cdot)\leq\zeta_2(\cdot)$, from Theorem \ref{th3} we have that $Y^{-\zeta_1(\cdot)} (t)\geq Y^{-\zeta_2(\cdot)} (t)$. Then,
\begin{align*}
\rho(t;\zeta_1(\cdot))=Y^{-\zeta_1(\cdot)} (t)\geq Y^{-\zeta_2(\cdot)} (t)=\rho(t;\zeta_2(\cdot)).
\end{align*}
$(iii)$ Assume that $(Y^{c}(\cdot), Z^{c}(\cdot,\cdot))$ is the solution of Eq. (\ref{44})
with respect to the terminal $-\zeta(\cdot)-c$. Thus,
\begin{equation}\label{48}
\begin{aligned}
&-\zeta(t)-c-\int_{t}^{T}\left[ \frac{r(s)}{2}\left(Y^{0} (s)-c e^{-\int_s^T r(u)du}+\mathbb{E}\left[ Y^{0} (s)-c e^{-\int_s^T r(u)du} \right]\right)-h(t,s,Z^0 (t,s))\right] ds\\
&+\int_{t}^{T} g(t,s,Z^0 (t,s))d\overleftarrow{B}(s)-\int_{t}^{T}Z^0 (t,s)d\overrightarrow{W}(s)\\
&=Y^{0}(t)-c+c \int_{t}^{T} \left(r(s)e^{-\int_{s}^{T} r(u) du}\right) ds
=Y^{0}(t)-c+c e^{-\int_{s}^{T} r(u) du} \left|_t^T \right.\\
&=Y^{0}(t)-c e^{-\int_{t}^{T} r(u) du}.
\end{aligned}
\end{equation}
Based on the uniqueness of solutions to Eq. (\ref{44}), (\ref{48}) implies that
\begin{align}\label{49}
Y^{c}(t)=Y^{0}(t)-c e^{-\int_{t}^{T} r(s) ds}, \  \ Z^{c}(t,s)=Z^{0}(t,s).
\end{align}
By the definition of $\rho$ and from (\ref{49}), we can obtain that
\begin{align*}
\rho(t;\zeta(\cdot)+c)=\rho(t;\zeta(\cdot))-c e^{-\int_t^T r(s)ds}, ~a.s., ~t\in[0,T].
\end{align*}
$(iv)$ For the convexity axiom, we shall prove that
\begin{align*}
\rho(t;\lambda\zeta_1(\cdot)+(1-\lambda)\zeta_2(\cdot))\leq\lambda\rho(t;\zeta_1(\cdot))+(1-\lambda)\rho(t;\zeta_2(\cdot)), ~a.s., ~t\in[0,T],~\lambda\in [0,1],
\end{align*}
which is equivalent to
\begin{align*}
Y^{-\left(\lambda\zeta_1(\cdot)+(1-\lambda)\zeta_2(\cdot)\right)}(t)\leq\lambda Y^{-\zeta_1(\cdot)}(t)+(1-\lambda)Y^{-\zeta_2(\cdot)}(t), ~a.s., ~t\in[0,T],~\lambda\in [0,1].
\end{align*}
Let $(\check{Y}(\cdot),\check{Z}(\cdot,\cdot))\in L_{\mathbb{F}}^{2} (0,T; \mathbb{R})\times L_{\mathbb{F}}^{2} (\Delta; \mathbb{R})$ be the solution of MF-BDSVIEs:
\begin{equation*}
\begin{aligned}
\check{Y} (t)&=-\lambda\zeta_1 (t)-(1-\lambda)\zeta_2 (t)+\int_{t}^{T} f (t, s, \check{Y} (s), \check{Z} (t, s), \mathbb{E}[\check{Y} (s)])ds\\
&~~~+\int_{t}^{T} g(t, s, \check{Z} (t, s)) d\overleftarrow{B}(s)
-\int_{t}^{T} \check{Z} (t, s) d\overrightarrow{W}(s), ~~t\in[0, T].
\end{aligned}
\end{equation*}
Denote
\begin{equation*}
\begin{aligned}
&\hat{Y} (t):=\lambda Y^{-\zeta_1 (\cdot)} (t)+(1-\lambda)Y^{-\zeta_2 (\cdot)} (t),\\
&\hat{Z} (t,s):=\lambda Z^{-\zeta_1 (\cdot)}(t,s)+(1-\lambda)Z^{-\zeta_2 (\cdot)}(t,s),
\end{aligned}
\end{equation*}
where $(Y^{-\zeta_i (\cdot)}(\cdot),Z^{-\zeta_i (\cdot)} (\cdot,\cdot))\in L_{\mathbb{F}}^{2} (0,T; \mathbb{R})\times L_{\mathbb{F}}^{2} (\Delta; \mathbb{R})~(i=1,2)$ satisfies
\begin{equation*}
\begin{aligned}
Y^{-\zeta_i (\cdot)} (t)&=-\zeta_i (t)+\int_{t}^{T} f (t, s, Y^{-\zeta_i (\cdot)} (s), Z^{-\zeta_i (\cdot)} (t, s), \mathbb{E}[Y^{-\zeta_i (\cdot)} (s)])ds\\
&~~~+\int_{t}^{T} g(t, s, Z^{-\zeta_i (\cdot)} (t, s)) d\overleftarrow{B}(s)
-\int_{t}^{T} Z^{-\zeta_i (\cdot)} (t, s) d\overrightarrow{W}(s), ~~t\in[0, T].
\end{aligned}
\end{equation*}
By conditions (\ref{50})-(\ref{51}) yield that
\begin{equation*}
\begin{aligned}
\hat{Y} (t)&=-\lambda\zeta_1 (t)-(1-\lambda)\zeta_2 (t)
+\int_{t}^{T} \left(\lambda f(t, s, Y^{-\zeta_1 (\cdot)} (s), Z^{-\zeta_1 (\cdot)} (t, s), \mathbb{E}[Y^{-\zeta_1 (\cdot)} (s)])\right.\\
&~~~\left.+ (1-\lambda)f(t, s, Y^{-\zeta_2 (\cdot)} (s), Z^{-\zeta_2 (\cdot)} (t, s), \mathbb{E}[Y^{-\zeta_2 (\cdot)} (s)]) \right)ds\\
&~~~+\int_{t}^{T} \left(\lambda g(t, s, Z^{-\zeta_1 (\cdot)} (t, s))+(1-\lambda)g(t, s, Z^{-\zeta_2 (\cdot)} (t, s))\right) d\overleftarrow{B}(s)
-\int_{t}^{T} \hat{Z} (t, s) d\overrightarrow{W}(s)\\
&\geq -\lambda\zeta_1 (t)-(1-\lambda)\zeta_2 (t)
+\int_{t}^{T} f(t, s, \hat{Y} (s), \hat{Z} (t, s), \mathbb{E}[\hat{Y} (s)])ds\\
&~~~+\int_{t}^{T} g(t, s, \hat{Z} (t, s))d\overleftarrow{B}(s)
-\int_{t}^{T} \hat{Z} (t, s) d\overrightarrow{W}(s).
\end{aligned}
\end{equation*}
Using Theorems \ref{th1}-\ref{th3}, one can derive that
\begin{align*}
\check{Y} (t)\leq\hat{Y} (t),~a.s.,~ t\in[0,T].
\end{align*}
In consequence,
\begin{align*}
\rho(t;\lambda\zeta_1(\cdot)+(1-\lambda)\zeta_2(\cdot))=\check{Y} (t)&\leq\hat{Y} (t)\\
&=\lambda Y^{-\zeta_1 (\cdot)} (t)+(1-\lambda)Y^{-\zeta_2 (\cdot)} (t)\\
&=\lambda \rho(t;\zeta_1(\cdot))+(1-\lambda)\rho(t;\zeta_2(\cdot)),~a.s.,~ t\in[0,T],
\end{align*}
which is the required (\ref{52}).\\
In the same way as in the proof of $(iv)$, we can obtain $(v)$-$(vi)$.
\end{proof}

\section*{Acknowledgements}
The authors would like to thank the editors and the anonymous referees for
their constructive discussions
and comments which lead to a significant improvement of this paper.
This research is supported by the National Natural Science Foundation of China
(Grant Nos. 12101218, 11971483), the National Social Science Foundation of China (Grant No. 20BTJ044) and the Provincial Natural Science Foundation of Hunan (Grant Nos. 2023JJ40203, 2023JJ30642).








\end{document}